\documentclass[pdflatex,sn-mathphys-num]{sn-jnl}

\usepackage{amsmath,amssymb,amsthm,mathtools,bm}
\usepackage{comment}
\usepackage{hyperref}

\usepackage{bbm}

\usepackage{graphicx}%
\usepackage{multirow}%
\usepackage{mathrsfs}%
\usepackage[title]{appendix}%
\usepackage{xcolor}%
\usepackage{textcomp}%
\usepackage{manyfoot}%
\usepackage{booktabs}%
\usepackage{algorithm}%
\usepackage{algorithmicx}%
\usepackage{algpseudocode}%
\usepackage{listings}%

\theoremstyle{thmstyleone}%
\newtheorem{theorem}{Theorem}
\newtheorem{proposition}[theorem]{Proposition}%

\theoremstyle{thmstyletwo}%
\newtheorem{example}{Example}%
\newtheorem{remark}{Remark}%

\theoremstyle{thmstylethree}%

\newtheorem{lemma}[theorem]{Lemma}

\newtheorem{assumption}[theorem]{Assumption}

\begin{document}

\title[Article Title]{Improving Cramér--Rao Bound And Its Variants: An Extrinsic Geometry Perspective}


\author*[]{\fnm{Sunder Ram} \sur{Krishnan}}\email{eeksunderram@gmail.com}



\affil[]{\orgdiv{Department of Computer Science and Engineering}, \orgname{Amrita Vishwa Vidyapeetham}, \orgaddress{\street{Amritapuri}, \city{Kollam}, \postcode{690525}, \state{Kerala}, \country{India}}}




\abstract{This work presents a geometric refinement of the classical Cramér--Rao bound (CRB) in the non-asymptotic regime by incorporating curvature-aware corrections based on the second fundamental form associated with the statistical model manifold. That is, our formulation shows that relying on the extrinsic geometry of the square root embedding of the manifold in the ambient Hilbert space comprising square integrable functions with respect to a fixed base measure offers a rigorous (and intuitive) way to improve upon the CRB and some of its variants, such as the Bhattacharyya-type bounds, that use higher-order derivatives of the log-likelihood. Precisely, the improved bounds in the latter case make explicit use of the elegant framework offered by employing the Fa\`a di Bruno formula and exponential Bell polynomials in expressing the jets associated with the square root embedding in terms of the raw scores. The interplay between the geometry of the statistical embedding and the behavior of the estimator variance is quantitatively analyzed in concrete examples, showing that our corrections can meaningfully tighten the lower bound, suggesting further exploration into connections with estimator efficiency in more general situations.}




\maketitle

\section{Introduction}

The Cramér-Rao Bound (CRB) is a foundational result in parametric estimation theory, specifying a lower limit on the variance of any (locally) unbiased estimator under appropriate regularity conditions, and underpins the concept of estimator efficiency \cite{rao,vaart,schervish}.
While widely used, the CRB may be loose in several practical settings -- such as nonlinear models, low signal-to-noise regimes, or finite sample sizes -- leading to extensive literature on improved lower bounds. Notable classical refinements include the Chapman-Robbins bound \cite{chapman_robbins} and the Ziv-Zakai bound \cite{ziv_zakai}, which generalize the CRB beyond the usual regularity assumptions, such as differentiability under the integral sign, or invoke optimal detection problems.
Bayesian and hybrid variants of the CRB have been developed to incorporate prior uncertainty that avoid costly matrix inversions \cite{hybrid_crb}. Computational savings were also achieved with constrained CRBs and singular Fisher information matrices in \cite{tune2012_constrained_crb}.

A line of work that is of relevance to us includes the classical \emph{Bhattacharyya inequality} (\cite{b1946}; see also \cite{ZS,l1998}) that generalizes the CRB by incorporating higher-order derivatives of the log-likelihood and is of importance when an estimator that is efficient in the sense of the CRB is not available. Conditions under which a Bhattacharyya bound is a valid lower bound for the variance were studied in \cite{seth1949variance} by Seth, and considerations regarding attainment of the Bhattacharyya bound can be found in the work of Fend \cite{fend1959attainment}. It is also important for us to note that the Bhattacharyya bound can be viewed as a hierarchy of improvements over the CRB -- an increasing sequence of variance bounds that converge to the estimator variance -- based on projections of the estimator error onto spans of higher-order derivatives of the (log)likelihood. Convergence of the sequence of bounds to the variance of the best unbiased estimator for certain functions of the underlying parameter was shown in \cite{blight1974convergence} by Blight and Rao when the sampling distribution is one among normal, binomial, Poisson, negative binomial, or gamma with scale parameter. The same result was extended by Ghosh and Sathe for all estimable functions of the parameter and all multiparameter exponential families \cite{ghosh1987convergence}, provided the true value of the parameter is an interior point of the natural parameter space. Similar considerations were also made later by Pillai and Sinha \cite{pillai}. However, these elegant formulations of the Bhattacharyya bound lack a geometric interpretation despite the Hilbert space foundation.
In short, these constructions remain intrinsically \emph{score-based} and extrinsic geometry of the embedded statistical manifold, which results in nontrivial correction if the estimator error does not lie in the span of the considered higher-order derivatives, is not accounted for; this is exactly what we point out in our research.

Information geometry offers an alternative framework in which statistical models are viewed as Riemannian manifolds endowed with the Fisher-Rao metric. This viewpoint allows estimator performance to be studied using geometric concepts such as curvature, geodesics, and connections \cite{amari}. A pioneering work of direct relevance to ours is that of Efron \cite{efron1975} who tried to quantify the proximity of an arbitrary one-parameter family to the standard exponential class of distributions, which is, in some sense, a geodesic in the space of all possible probability distributions. The exponential family enjoys nice statistical properties; for instance, the maximum likelihood estimator (MLE) is a sufficient statistic and, provided that we choose an appropriate function of the parameter to estimate, attains the CRB. Efron introduced the notion of ``statistical curvature" of a one-parameter family at a given point and showed that large curvature corresponds to a breakdown of the good properties just alluded to; for a specific example, the variance of the MLE exceeds the CRB in approximate proportion to the squared statistical curvature -- asymptotically, the MLE extracts all but a part of the information in a sample that is proportional to the curvature squared. To add some more detail, Efron considered the problem of estimating the scalar parameter in a curved exponential family, on the basis of an i.i.d. sample of size \(n\), using a squared error loss function to evaluate consistent and efficient estimators that are smooth functions of the sufficient statistic. There he proved that the \(O(\tfrac{1}{n^2})\) term in the variance asymptotics precisely involves the squared statistical curvature, linking it to Rao's influential second-order efficiency theory. It is pertinent to note that Efron's curvature is the curvature of an embedding associated with the particular way a parameter space is placed inside a natural, higher-dimensional parameter space, as measured by the second fundamental form. \\A related fact, in the context of deriving the asymptotic MSE of a bias-corrected first-order efficient estimator, is given in Theorem 4.4 of \cite{amari}, where the statistical curvature term is identified with the e-curvature of the model manifold. In a similar vein to Efron's work, Takeuchi and Akahira \cite{takeuchi2003second} rigorously analyzed the higher-order efficiency of sequential estimation procedures in the scalar parameter case (their results also hold for nonsequential estimation). Continuing along these lines, Okamoto et al. \cite{okamoto1991asymptotic} developed a higher-order asymptotic theory of sequential estimation in the framework of the geometry of multidimensional curved exponential families leading to a design principle of the second-order efficient sequential estimation procedure.\\We note with great interest that Jim Reeds, in a discussion on Efron's work, points out that he concludes statements about the coefficients of the asymptotic expansions of the variance, but not about the variance itself as captured in conclusions using the CRB and Bhattacharyya inequalities. This fact, the square root embedding, and considering the relationship between the log-likelihood and square root embeddings to evaluate the effectiveness, represent three important points of departure in our work from Efron's significant contribution and the other papers mentioned above in this body of literature.

Most other information geometry-based works tend to focus on intrinsic properties and divergence measures, as briefly discussed below.
Several papers have extended the CRB to intrinsic Riemannian settings. In an important contribution, Smith \cite{smith_intrinsic_crb} derived intrinsic versions of the CRB on manifolds utilizing an arbitrary affine connection with arbitrary geodesics for both biased and unbiased estimators while Barrau and Bonnabel \cite{boumal2013_intrinsic_crb}, in a follow-up, elucidated the extent to which the intrinsic CRB obtained by Smith is independent of a specific choice of coordinates. Further, the authors also addressed a question revolving around the dependence of the intrinsic bound on the underlying parameter. Extending these results to the Bayesian setting, Bouchard et al. \cite{bouchard2024} studied the problem of covariance matrix estimation when the data is sampled from a Gaussian distribution whose covariance matrix is drawn from an inverse Wishart distribution. Performance bounds for this problem were obtained, for both the MSE (Euclidean metric) and the natural Riemannian distance for Hermitian positive definite matrices (affine invariant metric), using the derived intrinsic Bayesian CRB. Interestingly, it was shown that the maximum a posteriori and the minimum MSE estimators are asymptotically efficient in the intrinsic case, but not so in the Euclidean situation. In a slightly different vein, Loutchko et al. \cite{loutchko2025} developed, among other results, a multivariate CRB by exploring the Riemannian geometrical aspects of the information geometry of chemical reaction networks.

Another important thread involves divergence-based generalizations of the CRB. Mishra and Kumar \cite{mishra2020} proposed a Riemannian metric based on relative $\alpha$-entropy to obtain a generalized Bayesian Cramér-Rao inequality, thereby establishing a lower bound for the variance of an unbiased estimator for the $\alpha$-escort distribution starting from an unbiased estimator for the underlying distribution. Subsequently, Kumar and Mishra \cite{ashok2020} provided a unifying differential geometric framework for deriving CRB-like inequalities from generalized Csiszár divergences. They applied Eguchi's theory to derive the Fisher information metric and dual affine connections arising from these generalized divergence functions, which enabled them to prove a more widely applicable version of the Cramér-Rao inequality. In \cite{mishra2021}, these authors again relied on Eguchi's theory of obtaining dualistic geometric structures from a divergence function and then applied Amari-Nagoaka's theory \cite{amari} to obtain a Cramér-Rao type inequality. These approaches serve to find unbiased and efficient estimators, but typically remain in an intrinsic, divergence-induced geometry.

However, we note that a conceptual gap remains. The algebraic projection-based CRBs, such as those of Bhattacharyya and some of the follow-ups we met above, do not exploit the geometric curvature of the statistical manifold. While it may be argued that the Bhattacharyya-type bounds following Efron's paper capture second-order curvature information, these are all asymptotic and do not capture the inherent geometry in non-asymptotic variance expressions. In contrast, many of the information geometry-based CRBs focus on intrinsic structures and rarely intersect with projection-theoretic Hilbert space methods. This motivates our research on incorporating \emph{extrinsic} geometric effects -- such as the second fundamental form -- into (non-asymptotic) variance bounds.

In recent years, there has been a rapidly growing body of work that develops estimation theory, efficiency notions, and algorithmic methods based on the \emph{Wasserstein geometry} of probability measures, complementing the classical Fisher--Rao framework. A central contribution in this direction is due to Li and Zhao \cite{Li}, who introduced the \emph{Wasserstein information matrix} (WIM) by pulling back the Otto--Wasserstein metric from the space of probability measures to finite-dimensional parametric statistical models. This construction yields Wasserstein analogues of score functions, covariance operators, and a Wasserstein--CRB, together with asymptotic efficiency results for Wasserstein natural gradient methods. Subsequent work has further clarified the statistical meaning and attainability of these bounds. In particular, Nishimori and Matsuda \cite{nishimori2025} studied conditions under which estimators attain the Wasserstein--CRB asymptotically, identifying Wasserstein analogues of exponential families and establishing Wasserstein efficiency for location--scale models. Closely related is the work of Garc\'ia Trillos, Jaffe, and Sen \cite{garciatrillos_wcrb}, who developed a detailed Wasserstein--Cram\'er--Rao theory for unbiased estimation, interpreting Wasserstein variance as a notion of estimator \emph{sensitivity} to infinitesimal perturbations of the data, and this insight allowed them to derive a collection of results which are analogous to the classical case; for instance, a characterization of models in which there exist unbiased estimators achieving the lower bound exactly. At a more structural level, Li and Zhao \cite{LiZhaoScalingLimitsWIM} investigate scaling limits of WIM on Gaussian mixture models, showing that suitably renormalized pullback Wasserstein metrics admit well-defined limits as component variances vanish. Their analysis yields generalized Wasserstein metrics on probability simplices, including second-order corrections and associated Wasserstein gradient flows. This line of work deepens the understanding of Wasserstein-induced geometry on parametric statistical models, but is primarily concerned with metric limits and dynamical behavior, rather than variance lower bounds or curvature-induced corrections to classical Fisher-based Cramér–Rao inequalities.

These Wasserstein-based approaches represent a substantial conceptual shift: rather than refining the classical variance bound associated with the Fisher--Rao (or Hellinger) geometry, they redefine the notion of estimator instability using transport-based metrics. As a consequence, the resulting bounds are not directly comparable to the classical Cram\'er--Rao or Bhattacharyya inequalities, but instead address robustness and sensitivity properties that are invisible to Fisher-based variance analysis.

Complementary to the above Wasserstein–Cramér–Rao–type developments, several works emphasize the role of the ambient metric in shaping representations and induced geometries on families of probability distributions. Arias-Castro and Qiao \cite{AriasCastroQiaoEmbedding} study the embedding of distributional data via classical scaling and Isomap, adapting tools from Euclidean dimensionality reduction to probability measures and explicitly highlighting how the choice of metric (in particular, Wasserstein-type metrics) governs the resulting geometry and statistical structure. While their focus is on representation and visualization rather than estimation-theoretic bounds, their results further underscore the central role played by the ambient geometry of distribution space.

A related line of work compares Fisher and Wasserstein geometries directly. Amari and Matsuda \cite{amari_matsuda_wasserstein} studied affine deformation models and showed that Wasserstein geometry naturally separates shape and deformation parameters, trading Fisher efficiency for robustness against waveform perturbations; they further demonstrated that Wasserstein estimators coincide with MLE only in the Gaussian case. Fukushi, Nakanishi-Ohno, and Matsuda \cite{fukushi2024flatness} investigated the intrinsic and extrinsic geometry of location--scale--shape models under the Wasserstein metric, showing that while such models may be intrinsically flat, they can still exhibit extrinsic curvature in the ambient Wasserstein space. These results highlight how extrinsic geometry plays a role in Wasserstein statistics, albeit in a setting fundamentally different from Fisher--Rao variance analysis.

From a computational and algorithmic perspective, Wasserstein geometry has also been extensively studied in connection with optimization, sampling, and gradient flows. Wu and Gardner \cite{wu_gardner_ngvi} analyzed stochastic natural gradient variational inference, establishing non-asymptotic convergence rates for conjugate models and clarifying the geometric limitations of natural gradient methods in non-conjugate settings. Zuo's recent dissertation \cite{zuo2025} further advances the computational theory of Wasserstein gradient flows, developing primal--dual damping algorithms, neural approximations of Wasserstein flows, and accelerated Langevin-type samplers. These works emphasize computational efficiency and dynamics in Wasserstein space, rather than statistical lower bounds for estimator variance.

Finally, Lawson et al.~\cite{lawson_fisher_geometry_normals} provided a comprehensive and practitioner-oriented account of Fisher geometry for multivariate normal distributions, focusing on geodesics, distances, and numerical computation of shortest paths. Their work elucidates the curvature induced by covariance structures and compares Fisher geodesics with alternative paths such as those arising from Wasserstein geometry, but does not address variance lower bounds or projection-based refinements of the Cram\'er--Rao inequality.

In contrast to the above Wasserstein-based and algorithmic frameworks, the present paper remains entirely within the classical Fisher--Rao variance paradigm and addresses a complementary question: how to \emph{sharpen non-asymptotic variance lower bounds} for (locally) unbiased estimators by exploiting the \emph{extrinsic geometry} of the statistical model embedded in an ambient Hilbert space. Our approach does not redefine variance or efficiency, nor does it replace the Fisher information metric. Instead, it refines the classical CRB and Bhattacharyya-type bounds using log-likelihood derivatives by quantifying the contribution of estimator error components orthogonal to the tangent space, using the second fundamental form of the square root embedding/normal components that lie beyond the span of the scores. To the best of our knowledge, this constitutes the first systematic use of extrinsic curvature to obtain explicit, non-asymptotic improvements over the CRB, the emphasis being on the qualification \emph{non-asymptotic}, thereby bridging projection-theoretic variance bounds and geometric information theory in a manner distinct from both intrinsic Fisher geometry and Wasserstein-based regularizations.

That is, in this work, we bridge the gaps identified above by developing a Hilbert space projection-based and Riemannian geometric refinement of classical and higher-order CRBs. Suppose the family of probability distributions $\{P_\theta\}$, parameterized by the scalar $\theta\in \Theta$, is absolutely continuous with respect to a fixed base measure \(\mu\) with a strictly positive density \(f(\cdot;\theta)\). Define the square root embedding \(s_\theta: \Theta\mapsto L^2(\mu)\) by \(s_\theta:=\sqrt{f(\cdot;\theta)}\), where \(L^2(\mu)\) is the space of square integrable functions with respect to \(\mu\), and consider the jets \[
\eta_k(\theta):=\partial_\theta^k s_\theta,\qquad k\ge1.
\] Starting from a projection-theoretic decomposition of the estimator error in the ambient space $L^2(\mu)$, we derive a residual component orthogonal to the span of \(\eta_1\). We then show that the contribution of this residual can be explicitly quantified using the second fundamental form of the statistical manifold embedded in $L^2(\mu)$, yielding a curvature-aware correction to the CRB. A higher-order version of this result is also stated. Further, a variance bound that explicitly betters Bhattacharyya-type results using derivatives of the log-likelihood is derived by analyzing the structure of the \(\eta_k\) in relation to the raw score and its derivatives \(\partial_\theta^k\log f(\cdot;\theta)\). Importantly, note that we refer to results that use \(\partial_\theta^k\log f(\cdot;\theta)\), or other derivatives such as the \(\eta_k\), for \(k\geq 1\) as Bhattacharyya-type bounds throughout the paper since the classical Bhattacharyya bound used \(\frac{\partial^k_\theta f}{f}\) instead.

Our contributions are as follows.
\begin{itemize}
 \item We derive extrinsic geometric refinement terms using the second fundamental form in the non-asymptotic regime, providing insight into how curvature controls estimator error beyond linear approximations, and establish corrections to the classical CRB -- see Theorem \ref{thm:m1-corrected}.
    \item We interpret the higher-order variance lower bound sequence obtained by Bhattacharyya \cite{b1946}, and other authors following him, as successive $L^2$ projections, laying the foundation for geometric corrections making use of higher derivatives of the square root embedding and curvature information ``beyond a certain order" (cf. Theorem \ref{thm:curvature-corrected}). Thus, our version of the Bhattacharyya-type bound offers a clean geometric explanation for the differences between successive bounds.
    \item By a careful analysis of the form of the jets \(\eta_k\) expressed in terms of the raw score and its derivatives \(\partial_\theta^k\log f(\cdot;\theta)\) via the Fa\`a di Bruno formula and the exponential Bell polynomials, we identify orthogonal components arising from the corresponding expansions. It is established that \(m\)-th order Bhattacharyya-type bounds using log-likelihood derivatives are improved by those that use the jets \(\{\eta_k\}_{k=1}^m\) when treating the aforementioned orthogonal components as extrinsic corrections; this result can be found in Proposition \ref{prop:single-Naug-relations}.
    \item The proposed framework is exemplified by applying it to explicit models, demonstrating strict (and often fully analytical) improvement over both classical CRBs and algebraic higher-order bounds. Specifically, we endeavor to illuminate how curvature plays a role in quantifying estimator inefficiency.
\end{itemize}

To sum up/reiterate, this work unifies algebraic and geometric approaches to variance lower bounds and represents, to the best of our knowledge, the first use of extrinsic curvature to sharpen the CRB in the non-asymptotic scenario. It opens new avenues in the intersection of information geometry, Hilbert space theory, and statistical estimation.

\section{Preliminaries}

In this section, we fix notation and the geometric framework, including square root embedding, connections, and second fundamental form, used in the sequel. The section is concluded with a brief sketch of our plan to use these ideas to refine the CRB and some higher-order versions of the kind worked out in \cite{b1946}.

\subsection{Notation, square root embedding, and assumptions}
Let \(\mathcal P=\{P_\theta:\theta\in\Theta\subset\mathbb R\}\) be a smooth parametric family of probability measures on a measurable space such that each \(P_\theta\) has a strictly positive density \(f(\cdot;\theta)\) with respect to a fixed base measure \(\mu\).

We work with a statistical manifold embedded in the fixed Hilbert space \(L^2(\mu)\) via the square root map.
Consider the square root embedding
\[
s:\Theta\to L^2(\mu),\qquad s(\theta):=s_\theta=\sqrt{f(\cdot;\theta)}.
\]
Denote the \(L^2(\mu)\) Hilbert space inner product, norm, and orthogonal projection onto a subspace \(V\) by
\[
\langle u,v\rangle := \int u(x)v(x)\,d\mu(x),\qquad \|u\|^2=\langle u,u\rangle, \text{ and }\mathrm{Proj}_V(\cdot).
\]
We have \(s_\theta\in L^2(\mu)\) and \(\|s_\theta\|^2=1\) for all \(\theta\). 
Expectation under \(P_\theta\) is denoted \(\mathbb E_\theta[\cdot]\).  For any measurable \(h\), we have \(\mathbb E_\theta[h(X)]=\int h(x)f(x;\theta)\,d\mu(x)=\langle h,s_\theta^2\rangle\).
 
\begin{assumption}
\label{ass:regularity}
Assume 
\[
\eta_k(\theta):=\partial_\theta^k s_\theta,\qquad k\ge1,
\] exist in the Fréchet sense in \(L^2(\mu)\). Further, these are assumed to be linearly independent.  Differentiation under the integral sign is permitted; with \(T(\cdot): \mathbb E_\theta[T]<\infty\),
\[
\frac{\partial}{\partial\theta}\int T(x) f(x;\theta)d\mu=\int T(x)\frac{\partial f(x;\theta)}{\partial\theta} d\mu,
\] 
whenever the right hand side is finite.
\end{assumption}
Define the (raw) score and higher derivatives of the log-density:
\[
Y_k(x;\theta):=\partial_\theta^k\log f(x;\theta),\qquad k\ge1,
\]
and the Fisher information \(\mathcal{I}(\theta):=\mathbb E_\theta[Y_1^2]\).\\
We have, in particular,
\begin{equation*}
\eta_1(\theta)=\partial_\theta s_\theta=\tfrac12 Y_1\,s_\theta,
\end{equation*}
and
\begin{equation*}
\langle\eta_1,\eta_1\rangle=\tfrac14\mathbb E_\theta[Y_1^2]=\tfrac14 \mathcal{I}(\theta).
\end{equation*}

The Fisher--Rao metric on the 1-dimensional parameter manifold is \(g(\partial_\theta,\partial_\theta)=\mathcal{I}(\theta)\), with \(\partial_\theta\) denoting the coordinate vector field. Equip \(s(\Theta)\) with the induced metric from \(L^2(\mu)\). The pullback metric from the square root embedding then equals \(\tfrac{1}{4} \mathcal{I}(\theta)\).

Fix \(m\ge1\).  Use Assumption \ref{ass:regularity} to define the finite-dimensional jet subspace
\[
\mathcal T_m(\theta):=\operatorname{span}\{\eta_1(\theta),\eta_2(\theta),\dots,\eta_m(\theta)\}\subset L^2(\mu),
\]
and the Gram matrix
\begin{equation}\label{eq:G-def}
G(\theta)=(G_{ij}(\theta))_{1\le i,j\le m},\qquad G_{ij}(\theta):=\langle\eta_i,\eta_j\rangle.
\end{equation}
By definition and assumed linear independence, the matrix \(G(\theta)\) is symmetric positive definite.

\subsection{Geometric preliminaries}
We now proceed to considering some fundamental geometric and miscellaneous preliminaries (the latter in the next subsection) that serve to make our presentation self-contained. \paragraph{Ambient (flat) derivative and induced connection}
As alluded to, we view \(L^2(\mu)\) as the ambient Hilbert space with the standard flat metric connection \(\nabla\); the ambient directional derivative in the coordinate direction \(\partial_\theta\) (that is, the pushforward of the coordinate field) acts in the natural way:
\[
\nabla_{\partial_\theta}\eta_j = \partial_\theta\eta_j = \eta_{j+1},\qquad j\ge1,
\]
We trivially have that \(\nabla\) satisfies metric-compatibility and torsion-freeness:
\[
X\langle U,V\rangle = \langle \nabla_X U, V\rangle + \langle U, \nabla_X V\rangle,
\qquad \nabla_X Y - \nabla_Y X = [X,Y],
\]
for arbitrary smooth ambient vector fields \(U,V,X,Y\).

Consider the \emph{induced connection} obtained by orthogonally projecting the ambient derivative onto \(\mathcal T_m\):
\begin{equation}\label{eq:induced-connection}
\nabla^{\mathrm{ind}}_{\partial_\theta}\eta_m:=\operatorname{Proj}_{\mathcal T_m}(\eta_{m+1}).
\end{equation}
Write the induced connection coefficients in the basis \(\{\eta_k\}_{k=1}^m\):
\[
\nabla^{\mathrm{ind}}_{\partial_\theta}\eta_m=\sum_{k=1}^m\Gamma^{\mathrm{proj}}_{k}\,\eta_k,
\]
where we have suppressed the dependence of the row vector \(\Gamma^{\mathrm{proj}}\), with components \(\Gamma^{\mathrm{proj}}_{k}\), on \(m\) (and ignored the risk of confusion!). With a similar notational ambiguity, define the row vector \(v(\theta)\) by
\begin{equation}\label{eq:A-def}
v_{j}(\theta):=\langle\eta_{m+1},\eta_j\rangle,\qquad 1\le j\le m.
\end{equation}
Solving the normal equations for the orthogonal projection yields the compact formula
\begin{equation}\label{eq:GammaProj}
\Gamma^{\mathrm{proj}}(\theta)=v(\theta)\,G(\theta)^{-1}.
\end{equation}
In particular, for \(m=1\) we have the identity
\begin{equation}
\label{eq:pro}
\Gamma^{\mathrm{proj}}=\frac{\langle\eta_2,\eta_1\rangle}{\langle\eta_1,\eta_1\rangle}.
\end{equation}
Again with \(m=1\), it is standard to show that the induced connection \(\nabla^{\mathrm{ind}}\) on the tangent space \(Ts(\Theta)\) obtained by orthogonally projecting the ambient derivative
  \[
  \nabla^{\mathrm{ind}}_X Y := \operatorname{Proj}_{Ts(\Theta)}\bigl(\nabla_X \widetilde Y\bigr),
  \]
  where \(\widetilde Y\) is any ambient extension of the vector field \(Y\), is well-defined; that is, it is independent of extension. It is also standard that it is torsion-free and metric-compatible with respect to the pull back metric, and hence coincides with the Levi--Civita connection of the immersed statistical manifold \(s(\Theta)\). We refer interested readers to standard references on Riemannian geometry such as \cite{lee2006riemannian}. The motivation behind \eqref{eq:induced-connection} for \(m>1\) should also now be clear.
\begin{remark}
\label{rem:l2}
Note that we work with the fixed Hilbert space \(L^2(\mu)\) because it provides a fixed inner product and a metric-compatible Levi--Civita connection. The ambient connection in \(L^2(P_\theta)\) is not compatible with the metric there. Of course, one could enforce metric-compatibility by hand (by “redefining the Christoffel symbols”), but then we are no longer working with the Levi--Civita connection of the ambient Hilbert space -- the geometry is compromised -- and we would be simply importing back the \(L^2(\mu)\) structure in disguise. The key role played by the square root map is also highlighted in our work \cite{Carkri} in the context of the associated jet bundle geometry. 
\end{remark}
\paragraph{Second fundamental form (vectors) and normal bundle}
\label{sec:II-def}
Define the ``second fundamental vector" by taking the normal component of the ambient derivative:
\begin{equation}\label{eq:II-components}
\mathrm{\Pi}_m:=\eta_{m+1}-\sum_{k=1}^m\Gamma^{\mathrm{proj}}_{k}\,\eta_k.
\end{equation}
For \(m=1\), this is just the standard second fundamental form \(\mathrm{II}(\eta_1,\eta_1)\). By construction, \(\mathrm{II}_m\in\mathcal T_m(\theta)^\perp\), i.e.
\[
\langle\mathrm{II}_m,\eta_i\rangle=0\qquad\text{for }1\le i\le m.
\]

Define the curvature (normal) span
\begin{equation}\label{eq:N_m-def}
\mathcal N_m(\theta):=\operatorname{span}\{\mathrm{II}_m\}\subset\mathcal T_m(\theta)^\perp.
\end{equation}
Note that \(\Pi_m\) may be interpreted as the remnant of the ``acceleration" \(\eta_{m+1}\) after subtracting out the ``tangential" part in \(\mathcal{T}_m\). A more formal presentation of the second fundamental vector may be found in our follow up \cite{Carkri}, where we cast it within the Cartan prolongation framework applied to the hierarchy of jet bundles. 

\subsection{Fa\`a di Bruno and exponential Bell polynomials}
\label{sec:faa-di-bruno-bell}

The combinatorics that repeatedly appear in our derivations, when expressing the jets \(\eta_j\) in terms of the raw scores \(Y_j\), are precisely encoded by the Fa\`a di Bruno formula and the exponential Bell polynomials. Below we state the identities in the form used in the paper and give a few expressions that result for relevant quantities such as \(v, G\), which in turn determine the induced connection via \(\Gamma^{\text{proj}}\), in terms of (mixed) moments of the \(Y_j\).\\
Recall \(s_\theta(x)=\sqrt{f(x;\theta)}\). For brevity, we suppress the dependence on \(x\) and \(\theta\) when there is no ambiguity. Set
\[
u(\theta):=\tfrac12\log f(\,\cdot\,;\theta),\qquad\text{so that}\qquad s_\theta=\exp\big(u(\theta)\big).
\]
Write \(u^{(r)}:=\partial_\theta^r u(\theta)\). Recall the (first-order) score \(Y_1=\partial_\theta \log f\) so that
\[
u^{(1)}=\tfrac12 Y_1,\quad u^{(2)}=\tfrac12 Y_2,\ \dots,\ u^{(k)}=\tfrac12 Y_k.
\]

\paragraph{Exponential (complete) Bell polynomials.}  
The \emph{complete exponential Bell polynomial} \(B_k(x_1,\dots,x_k)\) is defined for \(k\ge1\) by
\[
B_k(x_1,\dots,x_k)
\;=\; \sum_{m=1}^k B_{k,m}(x_1,\dots,x_{k-m+1}),
\]
where \(B_{k,m}\) are the \emph{partial} Bell polynomials. The explicit combinatorial formula for the partial Bell polynomials is
\begin{equation}\label{eq:partial-bell}
B_{k,m}(x_1,\dots,x_{k-m+1})
\;=\;
\sum_{\substack{(j_1,\dots,j_{k-m+1})\\[2pt] j_i\ge0}}
\frac{k!}{j_1! j_2!\cdots j_{k-m+1}!}\;
\prod_{i=1}^{\,k-m+1}\left(\frac{x_i}{i!}\right)^{j_i},
\end{equation}
the sum taken over all nonnegative integer solutions \((j_1,\dots,j_{k-m+1})\) of the two constraints
\[
\sum_{i=1}^{k-m+1} j_i \;=\; m,\qquad
\sum_{i=1}^{k-m+1} i\,j_i \;=\; k.
\]
Intuitively, each term of \(B_{k,m}\) corresponds to a partition of the set \(\{1,\dots,k\}\) into \(m\) blocks, where \(j_i\) counts the number of blocks of size \(i\).

\paragraph{Fa\`a di Bruno (exponential form).}  
A convenient form of Fa\`a di Bruno for the \(k\)-th derivative of an exponential composition \(s(\theta)=\exp(u(\theta))\) is
\begin{equation}\label{eq:faa-exp}
\frac{d^k}{d\theta^k}\, e^{u(\theta)}
\;=\; e^{u(\theta)}\, B_k\!\big(u^{(1)}(\theta),u^{(2)}(\theta),\dots,u^{(k)}(\theta)\big).
\end{equation}
Equivalently, with the partial Bell polynomials,
\[
\frac{d^k}{d\theta^k}\, e^{u(\theta)}
\;=\; e^{u(\theta)}\sum_{m=1}^k B_{k,m}\!\big(u^{(1)},u^{(2)},\dots,u^{(k-m+1)}\big).
\]
This identity is the standard Fa\`a di Bruno formula specialized to the case when the outer function in the composition is the exponential function.

\paragraph{Application to the square root embedding.}  
Apply \eqref{eq:faa-exp} with \(u(\theta)=\tfrac12\log f(\cdot;\theta)\). Since \(s_\theta=\exp(u(\theta))\), we obtain for every integer \(k\ge1\):
\begin{equation}\label{eq:eta-bell}
\eta_k(\cdot\,;\theta) \;=\; \partial_\theta^k s_\theta
\;=\; s_\theta\; B_k\!\big(u^{(1)},u^{(2)},\dots,u^{(k)}\big)
\;=\; s_\theta\; B_k\!\Big(\tfrac12 Y_1,\tfrac12 Y_2,\dots,\tfrac12 Y_k\Big).
\end{equation}
This is the compact identity of great relevance in the paper. The above expresses each \(k\)-th jet \(\eta_k\) as the product of the base function \(s_\theta\) and a finite linear combination of monomials of the form \(Y_{i_1}Y_{i_2}\cdots Y_{i_k}\), with coefficients determined by the Bell polynomials.

\paragraph{Computation of the Gram matrix \(G\) and vector \(v\).}
Using \eqref{eq:eta-bell} the Gram matrix entries are
\[
G_{ij} = \langle\eta_i,\eta_j\rangle = \int s_\theta^2 \; B_i\!\Bigl(\tfrac12 Y_1,\dots,\tfrac12 Y_i\Bigr)\;
B_j\!\Bigl(\tfrac12 Y_1,\dots,\tfrac12 Y_j\Bigr)\,d\mu,
\]
which is equivalently an expectation under \(P_\theta\):
\[
\quad G_{ij} \;=\; \mathbb E_\theta\!\Big[ B_i\bigl(\tfrac12 Y_1,\dots,\tfrac12 Y_i\bigr)
\; B_j\bigl(\tfrac12 Y_1,\dots,\tfrac12 Y_j\bigr)\Big] \quad.
\]
Similarly, the vector \(v\) defined in \eqref{eq:A-def} has entries
\[
v_{j} = \langle\eta_{m+1},\eta_j\rangle
= \mathbb E_\theta\!\Big[ B_{m+1}\bigl(\tfrac12 Y_1,\dots,\tfrac12 Y_{m+1}\bigr)
\; B_j\bigl(\tfrac12 Y_1,\dots,\tfrac12 Y_j\bigr)\Big].
\]
Both \(G\) and \(v\) are therefore computable from the joint moments of the vector \((Y_1,\dots,Y_{m+1})\).

\paragraph{Lower-order expansions.}  
For reader convenience, we record the first few cases obtained from \eqref{eq:eta-bell} by using the definitions \(B_1(x_1)=x_1\), \(B_2(x_1,x_2)=x_1^2+x_2\), \(B_3(x_1,x_2,x_3)=x_1^3+3x_1x_2+x_3\):
\begin{align}
\eta_1 &= s_\theta\;\tfrac12 Y_1, \label{eq:eta1}\\[4pt]
\eta_2 &= s_\theta\;\Big(\tfrac14 Y_1^2 + \tfrac12 Y_2\Big), \label{eq:eta2}\\[4pt]
\eta_3 &= s_\theta\;\Big(\tfrac18 Y_1^3 + \tfrac34 Y_1 Y_2 + \tfrac12 Y_3\Big). \label{eq:eta3}
\end{align}
With this, we may compute the entries of \(G\) and \(v\), and the resulting \(\mathrm{II}_m\). 

\medskip\noindent\textbf{Order 1:}
In the scalar \(m=1\) case, from \eqref{eq:eta1},
\[
G_{11} = \langle\eta_1,\eta_1\rangle = \tfrac14 \mathbb E_\theta[Y_1^2].
\]
\[
v_{1} = \langle\eta_2,\eta_1\rangle
= \Big\langle \bigl(\tfrac12 Y_2 + \tfrac14 Y_1^2\bigr)s_\theta, \tfrac12 Y_1 s_\theta\Big\rangle
= \tfrac14\mathbb E_\theta[Y_2Y_1] + \tfrac18\mathbb E_\theta[Y_1^3].
\]
Further,
\[
\Gamma^{\mathrm{proj}}_{1} = \frac{v_{1}}{G_{11}}.
\]
Therefore the second fundamental form equals
\begin{equation}
\label{eq:alt}
\mathrm{II}_1 = \eta_2 - \Gamma^{\mathrm{proj}}_{1}\,\eta_1
= s_\theta\Big(\tfrac12 Y_2 + \tfrac14 Y_1^2 - \frac{v_{1}}{G_{11}}\cdot \tfrac12 Y_1\Big).
\end{equation}

\medskip\noindent\textbf{Order 2:}
We now rely also on \eqref{eq:eta2} and \eqref{eq:eta3}. The relevant Gram entries are
\begin{align*}
G_{12} &= \langle\eta_1,\eta_2\rangle
= \tfrac14\mathbb E_\theta[Y_1Y_2] + \tfrac18\mathbb E_\theta[Y_1^3],\\
G_{22} &= \langle\eta_2,\eta_2\rangle
= \mathbb E_\theta\Big[\bigl(\tfrac12 Y_2 + \tfrac14 Y_1^2\bigr)^2\Big]
= \tfrac14\mathbb E_\theta[Y_2^2] + \tfrac14\mathbb E_\theta[Y_2Y_1^2] + \tfrac{1}{16}\mathbb E_\theta[Y_1^4].
\end{align*}
Similarly,
\[
v_{1}=\langle\eta_3,\eta_1\rangle=\tfrac14\mathbb E_\theta[Y_3Y_1]+\tfrac38\mathbb E_\theta[Y_1^2Y_2]+\tfrac{1}{16}\mathbb E_\theta[Y_1^4],
\]
and
\[
v_{2}=\langle\eta_3,\eta_2\rangle=\mathbb E_\theta\Big[\bigl(\tfrac12 Y_3 + \tfrac34 Y_1Y_2 + \tfrac18 Y_1^3\bigr)\bigl(\tfrac12 Y_2 + \tfrac14 Y_1^2\bigr)\Big],
\]
which also expands into a linear combination of mixed moments.

With the \(2\times2\) matrix \(G=[G_{ij}]\) and \(1\times 2\) vector \(v=[v_{j}]\), we may then compute the row vector \(\Gamma^{\mathrm{proj}}=v G^{-1}\).  Subsequently, using
\[
\mathrm{II}_2 = \eta_3 - \Gamma^{\mathrm{proj}}_{1}\,\eta_1 - \Gamma^{\mathrm{proj}}_{2}\,\eta_2,
\]
and substituting the exact Bell polynomial-based identities for \(\eta_1,\eta_2,\eta_3\) gives the fully explicit expansion of \(\mathrm{II}_2\).

\subsection{Outlook and road ahead}
\label{sub:road}
Consider any unbiased estimator \(T(X)\), where \(X\sim P_\theta\), of the scalar parameter \(\theta\) and define the centered estimation error \(Z_0:=T(X)-\theta\) viewed as an element of \(L^2(P_\theta)\) or, equivalently, focus on the vector \(\widetilde{Z}_0:=Z_0s_\theta\in L^2(\mu)\). Similarly, also define \(\widetilde{Y}_1:=Y_1s_\theta, \widetilde{Y}_2:=Y_2 s_\theta, \dots, \widetilde{Y}_m:=Y_m s_\theta\) and \(\widetilde{\mathcal{T}}_m:=\mathrm{span}\{\widetilde{Y}_1, \dots, \widetilde{Y}_m\}.\) 

Now, it is easy enough to show (we do it below) that \(\|\mathrm{Proj}_{\mathcal{T}_1}\widetilde{Z}_0\|^2=\|\mathrm{Proj}_{\widetilde{\mathcal{T}}_1}\widetilde{Z}_0\|^2=\tfrac{1}{\mathcal{I}(\theta)}\), which lower bounds the variance \(\mathrm{Var}_\theta[T]\). 
Our insight is so simple, hence elegant: Orthogonally decompose
\[
\widetilde{Z}_0=\operatorname{Proj}_{\mathcal T_m}\widetilde{Z}_0 + R_m,\qquad R_m\in\mathcal T_m^\perp.
\]
Then the squared norm \(\|\operatorname{Proj}_{\mathcal T_m}\widetilde{Z}_0\|^2\) is the projection-based portion of the variance (for \(m=1\) it reproduces the CRB, as just mentioned); the residual \(\|R_m\|^2\) is the unexplained variance by the \(m\)-term projection. Projecting \(R_m\) further onto the curvature span \(\mathcal N_m\) (see \eqref{eq:N_m-def}) yields an explicit curvature correction. In summary, we may refine the CRB by considering the curvature-induced residual for \(m=1\). We emphasize again that this idea is distinct from the method of Bhattacharyya \cite{b1946} which uses higher-order derivatives of the likelihood and does not exploit the \emph{extrinsic geometry}, as we do, via the square root embedding. Those works that adopt an extrinsic differential geometric viewpoint, such as those of Efron \cite{efron1975} (see also Theorem 4.4 in \cite{amari}), are statements about asymptotic variance bounds.

With \(m>1\), Bhattacharyya-type results in the literature \cite{b1946, ZS, l1998} can be considered as improving the CRB using the subspace \(\widetilde{\mathcal{T}}_m\), and \textit{not} \(\mathcal{T}_m=\mathrm{span}\{\eta_1,\dots,\eta_m\}\), in the projection above. In any case, we prefer to work with \(\mathcal{T}_m\) in developing our Bhattacharyya-type bound, which can be viewed as a curvature-correction theorem (cf. Theorem \ref{thm:curvature-corrected} below), for reasons related to those quoted in Remark \ref{rem:l2}; in short, this approach offers an intuitive and correct geometric picture. One may ask: How does the subspace \(\mathcal{T}_m\) compare with \(\widetilde{\mathcal{T}}_m\) or, equivalently, how does our Bhattacharyya-type theorem relate to those that appeared before, using log-likelihood derivatives? While the spaces \(\widetilde{\mathcal{T}}_1=\mathcal{T}_1\), the answer is a bit more subtle in the \(m>1\) case, and we refer interested readers to Section \ref{sec:tt} for some results that shed light on this scenario and Section \ref{sec:bhat} for further discussion. Our conclusion, which improves bounds of the nature of Bhattacharyya using log-likelihood, appears in Proposition \ref{prop:single-Naug-relations} where we specifically use the relationship between the \(\eta_j\) and \(Y_j\) given in \eqref{eq:eta-bell}.

\section{Curvature-corrected variance bounds -- CRB case}

We now state and prove the main bounds where, unless mentioned otherwise and as has been the case throughout, all inner products are those of the fixed Hilbert space \(L^2(\mu)\), and all objects depend on the true parameter value \(\theta\) (but we suppress the explicit \(\theta\) argument, as has been frequent practice, for notational brevity). Recall that we have established (cf. \eqref{eq:pro} or \eqref{eq:alt}) that the second fundamental form is the normal component of \(\eta_2\):
\[
\mathrm{II}(\eta_1,\eta_1) \;=\; \eta_2 - \frac{\langle\eta_2,\eta_1\rangle}{\langle\eta_1,\eta_1\rangle}\,\eta_1
\;\in\; \mathcal T_1^\perp.
\]
The proof of the \(m>1\) case in Theorem \ref{thm:curvature-corrected} essentially covers that of \(m=1\) as well, but we state the \(m=1\) case separately since it corresponds to addressing inefficiency with respect to the classical CRB.
\begin{theorem}[Residual variance and CRB refinement, \(m=1\)]
\label{thm:m1-corrected}
Assume the model satisfies Assumption \ref{ass:regularity}. Fix the true parameter value \(\theta\) and let \(T(X)\) be an estimator of \(\theta\) which is unbiased (in a neighbourhood of \(\theta\)). Define the centered error
\[
Z_0(X) := T(X)-\theta.
\]
With \(s_\theta=\sqrt{f(\cdot;\theta)}\in L^2(\mu)\), consider the ambient error vector, the one-dimensional tangent space, and residual, respectively:
\[
\widetilde{Z}_0 := Z_0\,s_\theta,
\quad
\mathcal T_1=\operatorname{span}\{\eta_1\}\subset L^2(\mu),\quad
R_1:=\widetilde{Z}_0 - \operatorname{Proj}_{\mathcal T_1}\widetilde{Z}_0 \in \mathcal T_1^\perp.
\]
Finally denote the Fisher information by \(\mathcal I(\theta)=\mathbb E_\theta[Y_1^2]\), where the raw score \(Y_1=\partial_\theta\log f(X;\theta)\).

Then the following statements hold.
\begin{enumerate}
  \item (Variance decomposition/CRB recovery)
  \[
  \operatorname{Var}_\theta[T]
  \;=\; \big\|\operatorname{Proj}_{\mathcal T_1}\widetilde{Z}_0\big\|^2 \;+\; \|R_1\|^2
  \;\ge\; \frac{1}{\mathcal I(\theta)}.
  \]
  Specifically, under the stated unbiasedness and regularity assumptions
  \[
  \big\|\operatorname{Proj}_{\mathcal T_1}\widetilde{Z}_0\big\|^2 \;=\; \frac{1}{\mathcal I(\theta)},
  \]
  so the classical Cram\'er--Rao lower bound is obtained from the \(m=1\) projection of the centered error onto the tangent space \(\mathcal{T}_1\).
  \item (Curvature lower bound)
  \begin{equation}\label{eq:thm1-corrected}
  \|R_1\|^2
  \;\ge\; \frac{\big\langle \widetilde{Z}_0,\; \mathrm{II}(\eta_1,\eta_1)\big\rangle^2}
  {\big\| \mathrm{II}(\eta_1,\eta_1)\big\|^2},
  \end{equation}
  with equality if and only if \(R_1\) is collinear (in \(L^2(\mu)\)) with \(\mathrm{II}(\eta_1,\eta_1)\).
\end{enumerate}
In particular, combining our conclusions gives the curvature-corrected bound
\[
\operatorname{Var}_\theta[T] \;=\; \|\widetilde{Z}_0\|^2
\;\ge\; \frac{1}{\mathcal I(\theta)} \;+\; \frac{\big\langle \widetilde{Z}_0,\; \mathrm{II}(\eta_1,\eta_1)\big\rangle^2}
{\big\| \mathrm{II}(\eta_1,\eta_1)\big\|^2}.
\]
\end{theorem}

\begin{proof}
1. Orthogonally decompose the ambient vector \(\widetilde{Z}_0\) into tangential and normal parts relative to \(\mathcal T_1\):
\[
\widetilde{Z}_0 = \operatorname{Proj}_{\mathcal T_1}\widetilde{Z}_0 + R_1,
\qquad \operatorname{Proj}_{\mathcal T_1}\widetilde{Z}_0 \in \mathcal T_1,\ R_1\in\mathcal T_1^\perp.
\]
Noting that \(\|\widetilde{Z}_0\|^2=\mathbb E_\theta[(T-\theta)^2]=\operatorname{Var}_\theta[T]\), taking squared norms, and using orthogonality yields the variance decomposition
\[
\operatorname{Var}_\theta[T] = \big\|\operatorname{Proj}_{\mathcal T_1}\widetilde{Z}_0\big\|^2 + \|R_1\|^2.
\]
The orthogonal projection onto the one-dimensional subspace \(\mathcal T_1=\operatorname{span}\{\eta_1\}\) has squared norm
\[
\big\|\operatorname{Proj}_{\mathcal T_1}\widetilde{Z}_0\big\|^2
=\frac{\langle \widetilde{Z}_0,\eta_1\rangle^2}{\langle\eta_1,\eta_1\rangle}.
\]
Using the identity \(\eta_1=\tfrac12 Y_1 s_\theta\) we compute
\[
\langle\eta_1,\eta_1\rangle = \frac14\,\mathbb E_\theta[Y_1^2] = \frac14\,\mathcal I(\theta),
\qquad
\langle \widetilde{Z}_0,\eta_1\rangle = \big\langle Z_0 s_\theta,\tfrac12 Y_1 s_\theta\big\rangle
= \tfrac12\,\mathbb E_\theta[(T-\theta)Y_1].
\]
Under the unbiasedness and regularity hypotheses we may differentiate \(\mathbb E_\theta[T]=\theta\) with respect to \(\theta\) and interchange derivative and expectation; this yields \(\mathbb E_\theta[T Y_1]=1\). Similarly, since \(f(\cdot;\theta)\) integrates to \(1\), we get \(\mathbb E_\theta[Y_1]=0\) and hence \(\mathbb E_\theta[(T-\theta)Y_1]=1\). Therefore
\[
\langle \widetilde{Z}_0,\eta_1\rangle = \tfrac12,
\quad\text{and}\quad
\big\|\operatorname{Proj}_{\mathcal T_1}\widetilde{Z}_0\big\|^2
= \frac{(1/2)^2}{\mathcal I(\theta)/4} = \frac{1}{\mathcal I(\theta)}.
\]

2. Since \(\mathrm{II}(\eta_1,\eta_1)\in\mathcal T_1^\perp\) and \(R_1\in\mathcal T_1^\perp\), both vectors lie in the same subspace \(\mathcal T_1^\perp\). In that subspace, the elementary one-dimensional projection inequality gives
\[
\|R_1\|^2 \;\ge\; \frac{\langle R_1,\mathrm{II}(\eta_1,\eta_1)\rangle^2}{\|\mathrm{II}(\eta_1,\eta_1)\|^2},
\]
with equality iff \(R_1\) is collinear with \(\mathrm{II}(\eta_1,\eta_1)\). Finally, because \(\operatorname{Proj}_{\mathcal T_1}\widetilde{Z}_0\) is orthogonal to \(\mathrm{II}(\eta_1,\eta_1)\), we have
\[
\langle R_1,\mathrm{II}(\eta_1,\eta_1)\rangle = \langle \widetilde{Z}_0,\mathrm{II}(\eta_1,\eta_1)\rangle,
\]
which yields \eqref{eq:thm1-corrected} and completes the proof.
\end{proof}
\section{Higher-order CRB and curvature corrections in the square root embedding}
\label{sec:ps-and-curvature}

Needless to (re)state, we keep Assumption \ref{ass:regularity} so that all quantities below are well defined in the fixed Hilbert space \(L^2(\mu)\).  We first state the Hilbert-space formulation of the Bhattacharyya bound in \(L^2(P_\theta)\) (as conceptualized in other works such as \cite{fend1959attainment}, for example) with the caveat that we use derivatives of the log-likelihood and not the ratios \(\frac{\partial_\theta^k f}{f}\), show its equivalent formulation in the ambient space \(L^2(\mu)\), and then present a basic curvature-corrected bound formulated relative to the jet subspace \(\mathcal T_m=\operatorname{span}\{\eta_1,\dots,\eta_m\}\); the core idea here is the same as for the \(m=1\) case. Taking into account the normal directions in \(\mathcal{T}_m^\perp\) other than \(\Pi_m\) (which uses \(\eta_{m+1}\)) in Proposition \ref{prop:single-Naug-relations}, we explicitly show that improvements over Bhattacharyya-type bounds using derivatives of the raw score of the same order \(m\) can be achieved by comparing the subspaces \(\mathcal{T}_m\) and \(\widetilde{\mathcal{T}}_m\).

\subsection{Bhattacharyya-type bounds in \(L^2(P_\theta)\) and its square root representation}
The following theorem rephrases the core insight of other authors (for instance, \cite{fend1959attainment, blight1974convergence, pillai}), who construct a sequence of variance lower bounds using higher-order score functions systematically covering the Bhattacharyya variants \cite{b1946}. We claim no real novelty in the following result, but include the same anyway for the following reasons: While their work asserts that equality in the $m$-th-order bound results if estimator error lies in the span of score functions up to order $m$, our formulation provides a characterization of equivalence using Hilbert space projections, and sets the stage for further improvements to be derived later via basic Riemannian geometric arguments.

As regards notation, consider $L^2(P_\theta)$ with inner product $\langle f, g \rangle_\theta = \mathbb{E}_\theta[f(X)g(X)]$ and norm $\|f\|_\theta^2 = \langle f, f \rangle_\theta$. We use $\operatorname{Proj}^\theta_{V}(\cdot)$ to denote orthogonal projection to a subspace $V$ of $L^2(P_\theta)$.
\begin{theorem}[Hilbert space characterization of higher-order CRB \cite{b1946}]
\label{thm:ps-L2P}
Let $T(X)$ be an unbiased estimator of a scalar parameter $\theta$, and define the centered error $Z_0 := T(X) - \theta \in L^2(P_\theta)$. For each $k\ge 1$, consider the $k$-th score function $Y_k = \partial_\theta^k \log f(X; \theta)$ ($\in L^2(P_\theta)$), and define the finite-dimensional subspace:
\[
\mathcal{T}^Y_m := \operatorname{span}\{Y_1, \dots, Y_m\} \subset L^2(P_\theta),
\]
where the $Y_i$ are assumed to be linearly independent and \(m\geq 1\). Let $G^Y \in \mathbb{R}^{m \times m}$ be the Gram matrix with entries $G^Y_{ij} := \langle Y_i, Y_j \rangle_\theta$, and define $b \in \mathbb{R}^m$ by $b_i := \langle Z_0, Y_i \rangle_\theta$. Denote the variance $\sigma^2 := \|Z_0\|_\theta^2 = \operatorname{Var}_\theta[T(X)]$, and define the matrix:
\[
B^{(m)} := \mathbb{E}_\theta \left[ \begin{bmatrix} Z_0 \\ Y_1 \\ \vdots \\ Y_m \end{bmatrix} \begin{bmatrix} Z_0 \\ Y_1 \\ \vdots \\ Y_m \end{bmatrix}^\top \right] = 
\begin{bmatrix}
\sigma^2 & b^\top \\
b & G^Y
\end{bmatrix}.
\]
Then we have $\operatorname{Var}_\theta[T(X)]\geq b^\top (G^{Y})^{-1} b= \|\operatorname{Proj}^\theta_{\mathcal{T}^Y_k}(Z_0)\|_\theta^2$.
Further, the following are equivalent:
\begin{enumerate}
    \item $Z_0 \in \mathcal{T}^Y_m$, i.e., the projection error vanishes: $\|Z_0 - \operatorname{Proj}^\theta_{\mathcal{T}^Y_m}(Z_0)\|_\theta^2 = 0$;
    \item $B^{(m)}$ is singular; specifically, $\left(\left[(B^{(m)})^{-1}\right]_{11}\right)^{-1}=0$ (under the convention that \(\tfrac{1}{\infty}=0\)).
\end{enumerate}

Define $\mathcal{B}^{(k)} := \|\operatorname{Proj}^\theta_{\mathcal{T}^Y_k}(Z_0)\|_\theta^2 $ for $k\geq 1$. If either of the above equivalent conditions hold, then $ \mathcal{B}^{(1)}\leq  \mathcal{B}^{(2)}\leq \dots$, and for all $k \ge m$,
    \[
    \mathcal{B}^{(k)} = \operatorname{Var}_\theta[T(X)].
    \]
\end{theorem}

\begin{proof}
Since $T(X)$ is unbiased, $Z_0 \in L^2(P_\theta)$ satisfies $\mathbb{E}_\theta[Z_0] = 0$ and its variance equals $\|Z_0\|_\theta^2$.\\ 
We can write:
\[
\operatorname{Proj}^\theta_{\mathcal{T}^Y_m}(Z_0) = \sum_{k=1}^m a_k Y_k, \quad \text{where } a = (G^Y)^{-1}b,
\]
and the squared projection norm is:
\[
\|\operatorname{Proj}^\theta_{\mathcal{T}^Y_m}(Z_0)\|_\theta^2 = b^\top (G^Y)^{-1} b.
\]
Thus, the residual variance is:
\[
0\leq \|Z_0 - \operatorname{Proj}^\theta_{\mathcal{T}^Y_m}(Z_0)\|_\theta^2 = \|Z_0\|_\theta^2 - b^\top (G^Y)^{-1} b = \sigma^2 - b^\top (G^Y)^{-1} b.
\]
\noindent
Now recall that the inverse of the block matrix $B^{(m)}$ has:
\[
\left[(B^{(m)})^{-1}\right]_{11} = \frac{1}{\sigma^2 - b^\top (G^Y)^{-1} b}= \frac{1}{\|Z_0 - \operatorname{Proj}^\theta_{\mathcal{T}^Y_m}(Z_0)\|_\theta^2},
\]
from the Schur complement of $G^Y$ in $B^{(m)}$. Therefore, the following are equivalent:\\
(i) \( Z_0 \in \mathcal{T}^Y_m \) or \( \|Z_0 - \operatorname{Proj}^\theta_{\mathcal{T}^Y_m}(Z_0)\|_\theta^2 = 0 \),\\
(ii) \( B^{(m)} \) is singular and \( \left(\left[(B^{(m)})^{-1}\right]_{11}\right)^{-1}=0 \) under the convention \(\frac{1}{\infty}=0\).

\noindent
Since \( Z_0 \in \mathcal{T}^Y_m \), projections onto a larger space, which are nondecreasing in norm until $k<m$, stabilize at $k=m$; i.e. for $k>m$, since \( \mathcal{T}^Y_k \supseteq \mathcal{T}^Y_m \), projections give the same result. Thus, \( \|\operatorname{Proj}^\theta_{\mathcal{T}^Y_k}(Z_0)\|_\theta^2 = \|Z_0\|_\theta^2 \) for all \( k \ge m \).
\end{proof}

\begin{remark}
For \(m=1\), it is obvious that the general bound \( b^\top (G^Y)^{-1} b \) becomes:
\[
\mathcal{B}^{(1)} = \frac{1^2}{\mathcal{I}(\theta)} = \frac{1}{\mathcal{I}(\theta)}.
\]
\end{remark}

\begin{lemma}[square root isometry and score-image representation]
\label{lem:sqrt-isometry}
Let \(s_\theta=\sqrt{f(\cdot;\theta)}\) and define the map \(\mathcal J:L^2(P_\theta)\to L^2(\mu)\) by \((\mathcal J h)(x):=h(x)\,s_\theta(x)\).   
For \(k\ge1\), define the score-image \(\widetilde{Y}_k:=\mathcal J(Y_k)=Y_k s_\theta\) and \(\widetilde Z_0:=\mathcal J(Z_0)=Z_0 s_\theta\).  Then
\[
\langle Y_i,Y_j\rangle_{\theta}=\langle\widetilde{Y}_i,\widetilde{Y}_j\rangle,\qquad
\langle Z_0,Y_i\rangle_{\theta}=\langle\widetilde Z_0,\widetilde{Y}_i\rangle.
\]
Consequently Theorem \ref{thm:ps-L2P} is equivalent to the identical projection statement in the fixed Hilbert space \(L^2(\mu)\) with subspace \(\widetilde{\mathcal{T}}_m:=\operatorname{span}\{\widetilde{Y}_1,\dots,\widetilde{Y}_m\}\). \\Specifically, $\mathcal{B}^{(k)}= \|\operatorname{Proj}^\theta_{\mathcal{T}^Y_k}(Z_0)\|_\theta^2 =\|\operatorname{Proj}_{\widetilde{\mathcal{T}}_k}(\widetilde{Z}_0)\|^2$, for $k\geq 1$.
\end{lemma}

\begin{proof}
The verification is immediate:
\[
\langle\mathcal J h_1,\mathcal J h_2\rangle=\int h_1(x)h_2(x)s_\theta(x)^2\,d\mu(x)=\int h_1(x)h_2(x)f(x;\theta)\,d\mu(x)=\langle h_1,h_2\rangle_{\theta}.
\]
Thus inner products and orthogonal projections are preserved under \(\mathcal J\), establishing the equivalence.
\end{proof}

\begin{example}[Saturation at \( m = 3 \) in a symmetric quartic location model]
\label{ex:sat3}
Consider the location model with density:
\[
f(x; \theta) = \frac{1}{Z} \exp\left( - (x - \theta)^4 \right), \quad x \in \mathbb{R}
\]
where the normalization constant is
\(
Z = \int_{-\infty}^{\infty} \exp(-u^4)\, du
\)
via the change of variable \( u = x - \theta \). This model is symmetric, smooth, and lies outside the exponential family. We work in the Hilbert space \(L^2(P_\theta)\).

Let \( X \sim f(x;\theta) \) and define the centered estimator:
\[
Z_0 = T(X) - \theta = X - \theta
\]

The first three (raw) score functions are:
\[
Y_1(x) = \frac{\partial}{\partial \theta} \log f(x;\theta) = 4(x - \theta)^3, \quad
Y_2(x) = \frac{\partial^2}{\partial \theta^2} \log f(x;\theta) = -12(x - \theta)^2
\]
\[
Y_3(x) = \frac{\partial^3}{\partial \theta^3} \log f(x;\theta) = 24(x - \theta)
\]

We compute the bounds:
\[
\mathcal{B}^{(1)} = \frac{ \langle Z_0, Y_1 \rangle_\theta^2 }{ \langle Y_1, Y_1 \rangle_\theta }, \quad
\mathcal{B}^{(2)} = 
\begin{bmatrix}
\langle Z_0, Y_1 \rangle_\theta & \langle Z_0, Y_2 \rangle_\theta
\end{bmatrix}
\cdot
\begin{bmatrix}
\langle Y_1, Y_1 \rangle_\theta & \langle Y_1, Y_2 \rangle_\theta \\
\langle Y_1, Y_2 \rangle_\theta & \langle Y_2, Y_2 \rangle_\theta
\end{bmatrix}^{-1}
\cdot
\begin{bmatrix}
\langle Z_0, Y_1 \rangle_\theta \\
\langle Z_0, Y_2 \rangle_\theta
\end{bmatrix}
\]
\[
\mathcal{B}^{(3)} = b^\top (G^Y)^{-1} b, \quad \text{with } b = 
\begin{bmatrix}
\langle Z_0, Y_1 \rangle_\theta \\
\langle Z_0, Y_2 \rangle_\theta \\
\langle Z_0, Y_3 \rangle_\theta
\end{bmatrix}, \quad
G^Y = 
\begin{bmatrix}
\langle Y_1, Y_1 \rangle_\theta & \langle Y_1, Y_2 \rangle_\theta & \langle Y_1, Y_3 \rangle_\theta \\
\langle Y_1, Y_2 \rangle_\theta & \langle Y_2, Y_2 \rangle_\theta & \langle Y_2, Y_3 \rangle_\theta \\
\langle Y_1, Y_3 \rangle_\theta & \langle Y_2, Y_3 \rangle_\theta & \langle Y_3, Y_3 \rangle_\theta
\end{bmatrix}.
\]

\textit{Symmetry-induced orthogonality.} Since the model is symmetric around \( \theta \), and \( Z_0, Y_1 \) are odd functions while \( Y_2 \) is even, it follows that:
\[
\langle Z_0, Y_2 \rangle_\theta = 0, \quad \langle Y_1, Y_2 \rangle_\theta = 0
\]
Therefore, we get:
\[
\mathcal{B}^{(2)} = \mathcal{B}^{(1)} < \operatorname{Var}_\theta[T]
\]

\textit{Third-order saturation.} Since \( Y_3(x) = 24(x - \theta) = 24 Z_0(x) \), we have:
\[
\langle Z_0, Y_3 \rangle_\theta = 24 \cdot \operatorname{Var}_\theta[T]
\]
and \( Z_0 \in \operatorname{span}\{Y_1, Y_3\} \), implying:
\[
\mathcal{B}^{(3)} = \operatorname{Var}_\theta[T]
\]

\textit{Numerical values.} Using numerical integration under the base measure \( p(u) = \frac{1}{Z} e^{-u^4} \), we obtain:
\[
\mathcal{B}^{(1)} = \mathcal{B}^{(2)} = 0.24656,\qquad
\mathcal{B}^{(3)} = \operatorname{Var}_\theta[T] = 0.33799
\]

This example demonstrates strict improvement of the lower bound at \( m = 3 \) and confirms saturation of the bound sequence at that order.
\end{example}

\subsection{Curvature-corrected Bhattacharrya-type bound, \(m>1\)}
We now state and prove the curvature-corrected bound, essentially our Bhattacharyya-type bound, formulated relative to the jet subspace \(\mathcal T_m\), and \textit{not} \(\widetilde{\mathcal T}_m\), since this really is the natural generalization of the \(m=1\) case (see also Remark \ref{rem:l2}). Thus we work in the fixed Hilbert space \(L^2(\mu)\) and the jet subspace \(\mathcal T_m\). The geometric refinement we use is again similar in spirit to the \(m=1\) scenario: project the estimator residual onto the jet space, then project the residual onto the curvature span inside the normal bundle. 

\begin{theorem}[Curvature-corrected variance bound, \(m\ge1\)]\label{thm:curvature-corrected}
Let \(T(X)\) be an unbiased estimator of \(\theta\) and \(Z_0=T(X)-\theta\). Define the centered error \(\widetilde Z_0:=Z_0 s_\theta\in L^2(\mu)\), where \(s_\theta=\sqrt{f(\cdot,\theta)}\). For \(k\ge1\), consider \(\eta_k=\partial_\theta^k s_\theta\) and the jet subspace
\[
\mathcal T_m = \operatorname{span}\{\eta_1,\dots,\eta_m\}\subset L^2(\mu).
\]
Consider the second fundamental vector
\[
\mathrm{II}_m:=\eta_{m+1} - \sum_{k=1}^m \Gamma^{\mathrm{proj}}_{k}\,\eta_k,
\]
where the row vector \(\Gamma^{\mathrm{proj}}=v G^{-1}\) with the matrix \(G_{ij}=\langle\eta_i,\eta_j\rangle\) and the row vector \(v_{j}=\langle\eta_{m+1},\eta_j\rangle\).  Define the curvature (normal) span
\[
\mathcal N_m := \operatorname{span}\{\mathrm{II}_m\}\subset \mathcal T_m^\perp.
\]
Then the following curvature-corrected variance bound holds:
\begin{equation}
\label{eq:curv-corr}
\operatorname{Var}_\theta[T] \ge \mathcal B^{(m)}_{\mathrm{jet}} + \|\operatorname{Proj}_{\mathcal N_m}(\widetilde Z_0)\|^2=:\mathcal{C}^{(m)}=\mathcal B^{(m+1)}_{\mathrm{jet}},
\end{equation}
where \(\mathcal B^{(m)}_{\mathrm{jet}}:=\|\operatorname{Proj}_{\mathcal T_m}\widetilde Z_0\|^2\) is the projection bound relative to the jet subspace.  Equality in \eqref{eq:curv-corr} holds iff \(R_m:=\widetilde{Z}_0 - \operatorname{Proj}_{\mathcal T_m}\widetilde{Z}_0\in\mathcal N_m\).
\end{theorem}

\begin{proof}
Start with the orthogonal decomposition \(\widetilde Z_0=\operatorname{Proj}_{\mathcal T_m}\widetilde Z_0 + R_m\), with \(R_m\in\mathcal T_m^\perp\).  The normal space \(\mathcal T_m^\perp\) contains the curvature span \(\mathcal N_m\) by construction, so we may orthogonally decompose the residual \(R_m\) with respect to \(\mathcal N_m\):
\[
R_m = \operatorname{Proj}_{\mathcal N_m}R_m + \bigl(R_m-\operatorname{Proj}_{\mathcal N_m}R_m\bigr),
\]
with the two terms orthogonal. Taking squared norms and summing orthogonal contributions yields\\
\(
\|\widetilde Z_0\|^2\)\[ = \|\operatorname{Proj}_{\mathcal T_m}\widetilde Z_0\|^2 + \|\operatorname{Proj}_{\mathcal N_m}R_m\|^2 + \|R_m-\operatorname{Proj}_{\mathcal N_m}R_m\|^2
\ge \|\operatorname{Proj}_{\mathcal T_m}\widetilde Z_0\|^2 + \|\operatorname{Proj}_{\mathcal N_m}R_m\|^2,
\]
which is \eqref{eq:curv-corr}.  Note that \(\operatorname{Proj}_{\mathcal N_m}R_m=\operatorname{Proj}_{\mathcal N_m}\widetilde Z_0\), since \(\operatorname{Proj}_{\mathcal T_m}\widetilde Z_0\perp\mathcal N_m\).  Equality occurs iff the orthogonal remainder \(R_m-\operatorname{Proj}_{\mathcal N_m}R_m\) vanishes, i.e. \(R_m\in\mathcal N_m\).
\end{proof}

\begin{remark}
There are two finite-dimensional subspaces inside \(L^2(\mu)\):
\[
\widetilde{\mathcal{T}}_k=\operatorname{span}\{\widetilde{Y}_1,\dots,\widetilde{Y}_k\},\qquad
\mathcal T_k=\operatorname{span}\{\eta_1,\dots,\eta_k\},
\]
which are not the same for \(k>1\). The Bhattacharyya-type projection bounds usually considered in the literature correspond to projecting the estimator residual onto \(\widetilde{\mathcal{T}}_k\) (equivalently, projecting \(Z_0\) onto \(\mathcal T_k^{Y}\) in \(L^2(P_\theta)\)). That projection yields \(\mathcal B^{(k)}_{\mathrm{score}}= \|\operatorname{Proj}_{\widetilde{\mathcal{T}}_k}\widetilde Z_0\|^2\).\\
Interpretation of incorporating higher-order derivatives in a Bhattacharyya-type bound as curvature corrections requires projecting onto the jet subspace \(\mathcal T_k=\operatorname{span}\{\eta_1,\dots,\eta_k\}\), which is the natural subspace to be considered with the induced connection and second fundamental vectors.  Even though \(\mathcal T_k, \,\widetilde{\mathcal{T}}_k\) are both \(k\)-dimensional subspaces under assumptions of linear independence, in general \(\mathcal T_k\neq\widetilde{\mathcal{T}}_k\) for \(k\ge2\), so \(\mathcal B^{(k)}_{\mathrm{jet}}\) and \(\mathcal B^{(k)}_{\mathrm{score}}\) are different quantities. Only when the linear spans exactly coincide do the two formulations agree.
The one-dimensional spaces \(\operatorname{span}\{\widetilde{Y}_1\}\) and \(\operatorname{span}\{\eta_1\}\) coincide because \(\eta_1=\tfrac12\widetilde{Y}_1\), which is why the classical CRB is unaffected by the choice. Nevertheless, it is obvious that our result in Theorem \ref{thm:curvature-corrected} offers a curvature-aware refinement of \(\mathcal B^{(m)}_{\mathrm{score}}\) when \(\mathcal B^{(m)}_{\mathrm{jet}}=\mathcal B^{(m)}_{\mathrm{score}}\) with \(m>1\). Can we compare them in the general scenario where \(\mathcal T_k\neq\widetilde{\mathcal{T}}_k\) for \(k\ge2\)? This is what we seek to answer in the next section.
\end{remark}

\subsection{Improving over Bhattacharyya-type bounds \(\mathcal B^{(m)}_{\mathrm{score}}\)}
\label{sec:tt}
What we proved in Theorem \ref{thm:curvature-corrected} is also a Bhattacharyya-type bound with the jets \(\eta_k\), which was interesting for the reason that, in the sequence of bounds \(\mathcal{B}_{\mathrm{jet}}^{(1)}\leq \mathcal{B}_{\mathrm{jet}}^{(2)}\leq \mathcal{B}_{\mathrm{jet}}^{(3)}\leq \dots\) (\(\mathcal{B}_{\mathrm{jet}}^{(1)}\) equalling the CRB), we get \[ \mathcal{B}_{\mathrm{jet}}^{(m+1)}-\mathcal{B}_{\mathrm{jet}}^{(m)}=\|\operatorname{Proj}_{\mathcal N_m}(\widetilde Z_0)\|^2,\] where the right side is precisely the curvature correction beyond order \(m\). An interesting question is: Staying with derivatives of order \(m\) and not considering the next higher derivative, can we demonstrate improvement over Bhattacharyya-type bounds that use the log-likelihood derivatives, i.e. \(\mathcal B^{(m)}_{\mathrm{score}}\)? For this purpose, we argue that it is beneficial to consider the structural relationship between the \(\eta_j\) and \(\widetilde{Y}_j\), subsequently incorporating normal directions in \(\mathcal T_m^\perp\) other than the single normal vector in \(\mathrm{span}\{\Pi_m\}\), which uses the \((m+1)\)-th order derivative. That is, our goal here is ultimately to compare the subspaces \(\mathcal{T}_m\) and \(\widetilde{\mathcal{T}}_m\) for \(m>1\). \\We start by noting the following, most of which we met earlier.
\paragraph{Setup and notation.}
Fix \(m> 1\). Let \(s_\theta\in L^2(\mu)\) be the square root embedding of the model and write
\(
\widetilde{Y}_j = Y_j\,s_\theta,\quad \eta_j = \partial_\theta^j s_\theta,\quad j\ge 1.
\)
Consider the jet subspace and the score-image span
\[
\mathcal T_m = \operatorname{span}\{\eta_1,\dots,\eta_m\},\qquad
\widetilde{\mathcal{T}}_m = \operatorname{span}\{\widetilde{Y}_1,\dots,\widetilde{Y}_m\}.
\]
For the purposes of this section, introduce compact notation for the projection operators \(P^\top:=\operatorname{Proj}_{\mathcal T_m}\), \(P^\perp:=\operatorname{Proj}_{\mathcal T_m^\perp}\).

By Faà di Bruno/Bell polynomials (cf. \eqref{eq:eta-bell}, and also \eqref{eq:eta1}, \eqref{eq:eta2}, \eqref{eq:eta3}) we have the jet decomposition
\begin{equation}\label{eq:etaPhiR}
\eta_j \;=\; \tfrac12\,\widetilde{Y}_j \;+\; W_j,\qquad j\ge2,
\end{equation}
where each remainder \(W_j\) is a finite linear combination of products (monomials) in the lower-order \(\widetilde{Y}_\ell\)'s (\(\ell<j\)). Set
\[
W_j^\perp := P^\perp W_j,\qquad j\ge2,
\]
the normal part of the remainder.

\paragraph{Augmented normal span.}
To capture normal directions produced by products in the Bell expansion define, for \(m\geq 2\), the augmented normal span
\begin{equation}\label{eq:Naug-single}
\mathcal N_m^{\mathrm{aug}} \;:=\; \operatorname{span}\{ W_2^\perp,\dots,W_m^\perp\}
\;\subseteq\; \mathcal T_m^\perp.
\end{equation}

\begin{lemma}[Containment of score span]
\label{lem:single-Naug-containment}
For every \(m\ge2\), we have the subspace inclusion
\begin{equation}\label{eq:S-in-TplusNaug}
\widetilde{\mathcal{T}}_m \;\subseteq\; \mathcal T_m \oplus \mathcal N_m^{\mathrm{aug}}.
\end{equation}
\end{lemma}

\begin{proof}
From \eqref{eq:etaPhiR} we get \(\widetilde{Y}_j=2\eta_j-2W_j\). Apply the orthogonal decomposition relative to \(\mathcal T_m\):
\[
\widetilde{Y}_j = P^\top\widetilde{Y}_j + P^\perp\widetilde{Y}_j = P^\top\widetilde{Y}_j - 2W_j^\perp,
\]
since \(P^\perp\eta_j=0\) for \(j\le m\). It is immediate that \(\widetilde{Y}_j\in\mathcal T_m\oplus\mathcal N_m^{\mathrm{aug}}\) for every \(j\le m\) (of course, \(\widetilde{Y}_1\in \mathcal T_m,\,\forall m\geq 1\)), and linear combinations of the \(\widetilde{Y}_j\)'s (i.e. \(\widetilde{\mathcal{T}}_m\)) belong to \(\mathcal T_m\oplus\mathcal N_m^{\mathrm{aug}}\), proving the containment.
\end{proof}

\begin{proposition}[Universal comparison with augmented normal span]
\label{prop:single-Naug-compare}
Let \(\widetilde{Z}_0=Z_0 s_\theta\in L^2(\mu)\), with notation as before. Then, for \(m\geq 2\),
\begin{equation}\label{eq:Caug-Bscore}
\mathcal C_{\mathrm{aug}}^{(m)} \;:=\; \big\|P^\top \widetilde{Z}_0\big\|^2 \;+\; \big\|\operatorname{Proj}_{\mathcal N_m^{\mathrm{aug}}}\widetilde{Z}_0\big\|^2
\;=\; \big\|\operatorname{Proj}_{\mathcal T_m\oplus\mathcal N_m^{\mathrm{aug}}}\widetilde{Z}_0\big\|^2
\;\ge\; \big\|\operatorname{Proj}_{\widetilde{\mathcal{T}}_m}\widetilde{Z}_0\big\|^2 \;=\; \mathcal B_{\mathrm{score}}^{(m)}.
\end{equation}
Equality holds iff \(\operatorname{Proj}_{(\mathcal T_m\oplus\mathcal N_m^{\mathrm{aug}})\ominus \widetilde{\mathcal{T}}_m}\widetilde{Z}_0=0\).
\end{proposition}

Notice that \(\mathcal C_{\mathrm{aug}}^{(m)}=\mathcal{B}^{(m)}_{\text{jet}}+\big\|\operatorname{Proj}_{\mathcal N_m^{\mathrm{aug}}}\widetilde{Z}_0\big\|^2\), and that one can compute \(\mathcal C_{\mathrm{aug}}^{(m)}\) only using the \(\{\eta_k\}_{k=1}^m\) offering a means to comparing bounds that can be obtained using the vectors \(\{\eta_k\}_{k=1}^m\) and \(\{\widetilde{Y}_k\}_{k=1}^m\). It is also obvious that \(\mathcal C_{\mathrm{aug}}^{(m)}\ge \max\{\mathcal{B}^{(m)}_{\text{jet}}, \mathcal{B}^{(m)}_{\text{score}}\}\).

\begin{proof}
Lemma \ref{lem:single-Naug-containment} gives \(\widetilde{\mathcal{T}}_m\subseteq\mathcal T_m\oplus\mathcal N_m^{\mathrm{aug}}\). For any closed subspaces \(V_1\subseteq V_2\) and any vector \(Z\), we have \(\|\operatorname{Proj}_{V_2} Z\|\ge\|\operatorname{Proj}_{V_1} Z\|\). Taking \(V_1= \widetilde{\mathcal{T}}_m\) and \(V_2=\mathcal T_m\oplus\mathcal N_m^{\mathrm{aug}}\) yields the inequality. Orthogonality of \(\mathcal T_m\) and \(\mathcal N_m^{\mathrm{aug}}\) gives the sum-of-squares decomposition on the left, and the equality criterion follows from standard projection geometry.
\end{proof}

\begin{proposition}[Strict improvement over \(\mathcal B_{\mathrm{score}}^{(m)}\)]
\label{prop:single-Naug-relations}
With notation as above,
 if \(\operatorname{Proj}_{\mathcal N_m^{\mathrm{aug}}}\widetilde{Z}_0\neq0\) and \(\mathcal N_m^{\mathrm{aug}}\not\subseteq \widetilde{\mathcal{T}}_m\), then strict improvement holds:
  \[
  \mathcal C_{\mathrm{aug}}^{(m)} \;>\; \mathcal B_{\mathrm{score}}^{(m)}.
  \]
\end{proposition}

\begin{proof}
By Proposition \ref{prop:single-Naug-compare}, \(\mathcal C_{\mathrm{aug}}^{(m)}=\|\operatorname{Proj}_{\mathcal T_m\oplus\mathcal N_m^{\mathrm{aug}}}\widetilde{Z}_0\|^2\ge\|\operatorname{Proj}_{\widetilde{\mathcal{T}}_m}\widetilde{Z}_0\|^2\). If \(\operatorname{Proj}_{\mathcal N_m^{\mathrm{aug}}}\widetilde{Z}_0\neq0\) and \(\mathcal N_m^{\mathrm{aug}}\not\subseteq \widetilde{\mathcal{T}}_m\), then the orthogonal complement \((\mathcal T_m\oplus\mathcal N_m^{\mathrm{aug}})\ominus\widetilde{\mathcal{T}}_m\) is nontrivial and contains a component of \(\operatorname{Proj}_{\mathcal N_m^{\mathrm{aug}}}\widetilde{Z}_0\); therefore projection onto the larger space strictly increases the squared norm and \(\mathcal C_{\mathrm{aug}}^{(m)}>\mathcal B_{\mathrm{score}}^{(m)}\).
\end{proof}

\begin{remark}
\label{rem:new}
It is worth repeating that, since the augmented normal space is defined as \(
 \operatorname{span}\{W_2^\perp,\dots,W_m^\perp\}
\;\subseteq\; \mathcal T_m^\perp,
\) \emph{without} the \(\Pi_m\) (cf. \eqref{eq:Naug-single}), we are not ``cheating" by considering a higher-order derivative in the comparison by including the normal vector \(\Pi_m\). In more formal words, Theorem \ref{thm:curvature-corrected}, which is our result of the nature of Bhattacharyya, showed that considering projections onto the new (and ``natural") normal direction \(\Pi_m\) is equivalent to incorporating a curvature-based correction, whereas Propositions \ref{prop:single-Naug-compare} and \ref{prop:single-Naug-relations} prove that improvement over \(m\)-th order Bhattacharyya-type bounds that use log-likelihood derivatives is possible using the specific structure of the \(\eta_j\) only upto order \(m\); specifically, we need to use the extrinsic directions \(W_2^\perp,\dots,W_m^\perp\). For the latter, it may be of interest for the reader to consider Example \ref{ex:asym-quartic-aug}.
\end{remark}

\subsection{Summary and comparison with Bhattacharyya-type bounds}
\label{sec:bhat}

The classical \emph{Bhattacharyya inequality} (\cite{b1946}; see also \cite{ZS, l1998}) generalizes the CRB by incorporating higher-order derivatives of the model density. 
The core idea here can be systematically understood within a projection framework in $L^2(P_\theta)$. 
One considers the span $\widetilde{\mathcal{T}}_m = \operatorname{span}\{\widetilde{Y}_1,\dots,\widetilde{Y}_m\}$ of score derivatives and projects the centered estimator residual $\widetilde Z_0 = (T - \theta)s_\theta$ onto $\widetilde{\mathcal{T}}_m$. 
This leads to an explicit sequence of bounds 
\[
\mathcal{B}_{\mathrm{score}}^{(m)} \;=\; \| \mathrm{Proj}_{\widetilde{\mathcal{T}}_m} \widetilde Z_0 \|^2,
\]
which coincides with bounds of the kind obtained by Bhattacharyya, but now using derivatives of the score (not \(\frac{\partial^k_\theta f}{f}\)), and recovers the CRB at $m=1$. 

In contrast, we considered refinements obtained by embedding the statistical model manifold extrinsically into the Hilbert space $L^2(\mu)$ using the square root map and exploiting the associated differential geometry. 
For $m=1$, we showed that the CRB may be refined by an explicit \emph{curvature correction} term involving the second fundamental form $\Pi(\eta_1,\eta_1)$, where $\eta_1$ is the first-order jet direction (see Theorem \ref{thm:m1-corrected}). 
This yields a strictly sharper bound whenever the estimator residual has nontrivial projection in the normal direction of the model manifold; that is, whenever we have a CRB-inefficient estimator --- considerations of this kind are not made in the non-asymptotic Bhattacharyya-type bounds. 

What we have shown may be succintly stated this way: Consider the jet space $\mathcal{T}_m = \mathrm{span}\{\eta_1,\dots,\eta_m\}$ of embedded derivatives of the statistical model. Defining, for \(m\geq 1\), \(
\mathcal{B}_{\mathrm{jet}}^{(m)} \;=\; \| \mathrm{Proj}_{\mathcal{T}_m} \widetilde Z_0 \|^2,
\)
we get the sequence of bounds \(\mathcal{B}_{\mathrm{jet}}^{(1)}\leq \mathcal{B}_{\mathrm{jet}}^{(2)}\leq \mathcal{B}_{\mathrm{jet}}^{(3)}\leq \dots\) (\(\mathcal{B}_{\mathrm{jet}}^{(1)}\) equalling the CRB), with \[ \mathcal{B}_{\mathrm{jet}}^{(m+1)}-\mathcal{B}_{\mathrm{jet}}^{(m)}=\|\operatorname{Proj}_{\mathcal N_m}(\widetilde Z_0)\|^2,\] where \(\|\operatorname{Proj}_{\mathcal N_m}(\widetilde Z_0)\|^2\) is precisely the curvature-based correction beyond order \(m\); see \eqref{eq:curv-corr}. These may be viewed as Bhattacharyya-type bounds using jets \(\eta_m\), but the interesting part is that they come with an associated geometric picture based on curvature; also see \cite{Carkri} for more differential geometric intuition.

Further, we considered how one may improve the log-likelihood based bounds \(\mathcal{B}_{\mathrm{score}}^{(m)}\) using only the \(\eta_j\) upto order \(m\). For $m>1$, we defined the augmented normal span
\[
\mathcal{N}^{\mathrm{aug}}_m = \mathrm{span}\big\{ W_2^\perp,\dots,W_m^\perp \big\},
\]
where the $W_j^\perp$ are orthogonal components arising from Faà~di~Bruno expansions of the embedding to facilitate comparing \(\mathcal{T}_m\) with \(\widetilde{\mathcal{T}}_m\). 
Our bound then reads
\[
\mathcal{C}_{\mathrm{aug}}^{(m)} 
\;=\; \big\| \mathrm{Proj}_{\mathcal{T}_m \oplus \mathcal{N}^{\mathrm{aug}}_m}\, \widetilde Z_0 \big\|^2,
\]
and we established that 
\[
\mathcal{C}_{\mathrm{aug}}^{(m)} =\mathcal{B}^{(m)}_{\text{jet}}+\big\|\operatorname{Proj}_{\mathcal N_m^{\mathrm{aug}}}\widetilde{Z}_0\big\|^2\;\ge\; \mathcal{B}_{\mathrm{score}}^{(m)},
\]
with strict inequality in generic curved models (see sufficient conditions stated in Proposition \ref{prop:single-Naug-relations}). 
Thus, the classical Bhattacharyya-type sequence may be viewed as a \emph{score-only case}, while our geometric refinement yields systematically sharper bounds by incorporating normal/extrinsic components that lie beyond the score span. 

To conclude, our extrinsic geometry point of view proposes curvature-based refinements to the classical CRB whenever the estimator error does not purely lie in the span of the score. A Bhattacharyya-type extension of this was also shown, and it was argued that this intuitive geometric picture does not present itself when using the score functions associated with the (log)likelihood. Although second-order information, as present in the curvature via the second fundamental form, is taken into account in some sense in the second-order asymptotic Bhattacharyya-type bounds \cite{efron1975,okamoto1991asymptotic}, this does not translate to non-asymptotic variance bounds.

\paragraph{Limitations of our approach.} Having clarified our contributions, we also acknowledge the limitations. Our work largely arose from a theoretical/conceptual motivation, seeking to put forth the point that a part of the inefficiency of estimators in the sense of classical statistical bounds may be understood in terms of a refinement based on extrinsic geometry of the statistical manifold in a fixed ambient \(L^2\) space. Some of the ideas developed here will be illustrated in the analytical examples below, which raise questions regarding ease of computation of the refinements in real situations. Our submission is that the bounds might involve more involved integral computations than when using the (log)likelihood-based scores, given the form of the \(\eta_j\), but they can be approximated numerically in many examples of interest without much difficulty. In fact, we do just that in the asymmetric quartic location family example below rather than deriving the full set of formulae by hand, where we seek to illustrate Proposition \ref{prop:single-Naug-relations}; see Example \ref{ex:asym-quartic-aug}. Our Bhattacharyya-type bound using the jets associated with \(s_\theta\) equals \[
\|\operatorname{Proj}_{\mathcal{T}_m}(\widetilde{Z}_0)\|_\theta^2 = b_\eta^\top (G)^{-1} b_\eta,
\] where the \(i\)-th component of \(b_\eta\) is \(\langle \widetilde{Z}_0,\eta_i \rangle\) and \(G\) is the Gram matrix associated with the \(\{\eta_i\}_{i=1}^m\). Possibly, this is trickier to compute than the bounds using log-likelihood derivatives because each \(\eta_i\) is a polynomial in the \(\{Y_j\}_{j\leq i}\) (cf. \eqref{eq:eta-bell}), but this complexity of the \(\eta_j\) is also what leads to improved bounds in certain models such as the one in Example \ref{ex:asym-quartic-aug}. Using likelihood derivatives \(\frac{\partial^i f}{f}\) as in the original Bhattacharyya bounds is easier to handle if one has sufficient regularity such that we may differentiate under the integral with impunity. In any case, as alluded to, our main motivation is to highlight the geometric viewpoint as regards inefficiency in non-asymptotic regimes leading to our choice of fully analytical examples below that serve this purpose.

\section{Analytical examples}
We illustrate Theorems \ref{thm:m1-corrected}, \ref{thm:curvature-corrected}, and some of the key results stated in Section \ref{sec:tt} with two fully analytical examples below and one that also involves some numerics; the first two analytical examples explicitly demonstrate curvature corrections to the CRB in the event of inefficiency (i.e. \(\mathcal{B}^{(1)}_{\text{score}}=\mathcal{B}^{(1)}_{\text{jet}}\equiv\mathcal B^{(1)}\)), while the second also illuminates results comparing \(\mathcal{B}_{\mathrm{score}}^{(m)}\) with our jet subspace-based curvature correction theorems. The third is specifically used to check our claim in Proposition \ref{prop:single-Naug-relations} regarding the fruitfulness of incorporating alternative extrinsic directions in the normal span in addition to the \(\mathcal{B}_{\mathrm{jet}}^{(m)}\) facilitating comparison and improvements over \(\mathcal{B}_{\mathrm{score}}^{(m)}\).

\begin{example}[One-parameter curved normal family]
We consider the family of normal laws, where \(P_\theta\) is the distribution
\[
\mathcal N\!\bigl(\mu(\theta),\,\sigma^2(\theta)\bigr),\qquad
\mu(\theta)=\theta,\qquad \sigma^2(\theta)=1+\theta^2,\qquad \theta\in\mathbb R,
\]
with Lebesgue base measure $\mu$ and strictly positive density
\[
f(x;\theta)=\frac{1}{\sqrt{2\pi(1+\theta^2)}}\,
\exp\!\left(-\frac{(x-\theta)^2}{2(1+\theta^2)}\right).
\]
Write $B:=1+\theta^2$ and $U:=X-\theta$ so that, under $P_\theta$, we have $U\sim \mathcal N(0,B)$.
The score and its first derivative are
\begin{align}
\label{eq:Y1}
Y_1(x;\theta)&=\partial_\theta\log f(x;\theta)
= -\frac{\theta}{B}+\frac{U}{B}+\frac{\theta U^2}{B^2},
\\
\label{eq:Y2}
Y_2(x;\theta)&=\partial_\theta^2\log f(x;\theta)
= -\frac{2}{B^2}-\frac{4\theta}{B^2}U+\frac{1-3\theta^2}{B^3}U^2.
\end{align}
Recall the jet fields
\[
\eta_1=\partial_\theta s_\theta=\frac12\,Y_1\,s_\theta,\qquad
\eta_2=\partial_\theta^2 s_\theta=\frac12\Bigl(Y_2+\tfrac12 Y_1^2\Bigr)s_\theta.
\]
Consider an unbiased estimator $T(X)$ of $\theta$. Set $Z_0=T(X)-\theta$ and identify
$Z_0 s_\theta\in L^2(\mu)$. For the location-equivariant unbiased estimator $T(X)=X$, we have
$Z_0=U$ and $\operatorname{Var}_\theta[T]=\mathbb E_\theta[U^2]=B$.

Using $U\sim\mathcal N(0,B)$ and the Gaussian moment formulae
$\mathbb E_\theta[U]=0$, $\mathbb E_\theta[U^2]=B$, $\mathbb E_\theta[U^4]=3B^2$, $\mathbb E_\theta[U^6]=15B^3$, $\mathbb E_\theta[U^8]=105B^4$, one obtains:
\begin{align}
\label{eq:Itheta}
\mathcal I(\theta)=\mathbb E_\theta[Y_1^2]
&=\frac{1+3\theta^2}{B^2},
\\
\label{eq:EY1Y2}
\mathbb E_\theta[Y_1 Y_2]
&=-\,\frac{2\theta(1+5\theta^2)}{B^3},
\\
\label{eq:EY2sq}
\mathbb E_\theta[Y_2^2]
&=\frac{3+10\theta^2+43\theta^4}{B^4},
\\
\label{eq:EY1cube}
\mathbb E_\theta[Y_1^3]
&=\frac{6\theta+14\theta^3}{B^3},
\\
\label{eq:UY2}
\mathbb E_\theta[U\,Y_2]
&=-\,\frac{4\theta}{B},
\\
\label{eq:UY1sq}
\mathbb E_\theta[U\,Y_1^2]
&=\frac{4\theta}{B}.
\end{align}

The Gram matrix $G\in\mathbb R^{2\times2}$ has entries $G_{ij}=\langle \eta_i,\eta_j\rangle$:
\begin{align}
\label{eq:G11}
G_{11}&=\left\langle\frac12 Y_1 s_\theta,\frac12 Y_1 s_\theta\right\rangle
=\frac14\,\mathbb E_\theta[Y_1^2]
=\frac{1+3\theta^2}{4B^2},
\\
\label{eq:G12}
G_{12}&=\left\langle\eta_1,\eta_2\right\rangle
=\frac14\left(\mathbb E_\theta[Y_1Y_2]+\frac12\mathbb E_\theta[Y_1^3]\right)
=\frac{\theta(1-3\theta^2)}{4B^3},
\\
\label{eq:G22}
G_{22}&=\left\langle\eta_2,\eta_2\right\rangle
=\frac14\left(\mathbb E_\theta[Y_2^2]+\mathbb E_\theta[Y_2Y_1^2]+\frac14\mathbb E_\theta[Y_1^4]\right)
=\frac{19+34\theta^2+75\theta^4}{16\,B^4},
\end{align}
where we used, in addition to \eqref{eq:EY2sq}, the identities
\[
\mathbb E_\theta[Y_2 Y_1^2]
=\frac{-55\theta^4-18\theta^2+1}{B^4},
\qquad
\mathbb E_\theta[Y_1^4]
=\frac{3\,(41\theta^4+22\theta^2+1)}{B^4}.
\]

For $Z_0 s_\theta$, the coefficients against $\eta_1,\eta_2$ are
\begin{align}
\label{eq:b1b2}
a_1&:=\langle Z_0 s_\theta,\eta_1\rangle=\frac12\,\mathbb E_\theta[Z_0 Y_1]
=\frac12,
\\
a_2&:=\langle Z_0 s_\theta,\eta_2\rangle
=\frac12\left(\mathbb E_\theta[Z_0 Y_2]+\frac12\mathbb E_\theta[Z_0 Y_1^2]\right)
=-\frac{\theta}{B},
\end{align}
where we used $\mathbb E_\theta[Z_0 Y_1]=1$, and \eqref{eq:UY2}–\eqref{eq:UY1sq} with $Z_0=U$.

The CRB ($m=1$) reads
\[
\mathcal{B}^{(1)}_{\text{score}}=\mathcal{B}^{(1)}_{\text{jet}}\equiv\mathcal B^{(1)}=\frac{1}{\mathcal I(\theta)}
=\frac{B^2}{1+3\theta^2}
=\frac{(1+\theta^2)^2}{1+3\theta^2}.
\]
The Christoffel coefficient is
\[
\Gamma_{11}=\frac{\langle \eta_2,\eta_1\rangle}{\langle \eta_1,\eta_1\rangle}
=\frac{G_{12}}{G_{11}}
=\frac{-3\theta^3+\theta}{3\theta^4+4\theta^2+1}.
\]
The second fundamental form is the normal component
\[
\mathrm{II}(\eta_1,\eta_1)
=\eta_2-\Gamma_{11}\,\eta_1 \;\in\; \mathcal T_1^\perp.
\]
Its squared norm and correlation with $Z_0 s_\theta$ are
\begin{align}
\label{eq:IInorm}
\|\mathrm{II}(\eta_1,\eta_1)\|^2
&=G_{22}-\Gamma_{11}^2 G_{11}
=\frac{189\theta^6 + 201\theta^4 + 87\theta^2 + 19}
{48\theta^{10} + 208\theta^8 + 352\theta^6 + 288\theta^4 + 112\theta^2 + 16},
\\
\label{eq:ZI}
\langle Z_0 s_\theta,\mathrm{II}(\eta_1,\eta_1)\rangle
&=a_2-\Gamma_{11} a_1
=-\frac{3\theta}{2\,(3\theta^2+1)}.
\end{align}
Therefore, our curvature-corrected lower bound for $m=1$ is
\begin{equation}
\label{eq:curv-corr-m1}
\operatorname{Var}_\theta[T]
\;\ge\; \mathcal B^{(1)}
+\frac{\langle Z_0 s_\theta,\mathrm{II}(\eta_1,\eta_1)\rangle^2}
{\|\mathrm{II}(\eta_1,\eta_1)\|^2}
=\frac{B^2}{1+3\theta^2}
+\frac{36\,\theta^2\left(\theta^8+4\theta^6+6\theta^4+4\theta^2+1\right)}
{567\theta^8 + 792\theta^6 + 462\theta^4 + 144\theta^2 + 19},
\end{equation}
and we get a strictly positive curvature correction term for all \(\theta\neq 0\). Of course, when \(\theta= 0\), \(\operatorname{Var}_\theta[T]=\mathcal B^{(1)}=1\) already.
\end{example}

\begin{example}[Strict gains at $m=1$ and $m=2$]
\label{ex:scalar-C1C2-gain}

Let $s_0(x)=(2\pi)^{-1/4}e^{-x^2/4}$ be the Gaussian square root and let $H_n$ be the probabilists' Hermite polynomials with
$\langle H_m s_0, H_n s_0\rangle = \mathbb E_\varphi[H_m H_n] = n!\,\delta_{mn}$ under $\varphi\sim\mathcal N(0,1)$.

Consider the one-parameter normalized family
\[
f_\theta(x) = \frac{s_0(x)^2 \exp\big(2 a_1(\theta) H_1(x) + 2 a_2(\theta) H_2(x)\big)}{Z(\theta)},\quad
s_\theta(x) = \sqrt{f_\theta(x)}
,
\]
with
\[
a_1(\theta)=\theta,\qquad a_2(\theta)=\theta+\tfrac{c}{2}\theta^2,\qquad
Z(\theta) = \int_\mathbb{R} s_0(x)^2 \exp\big(2 a_1(\theta) H_1 + 2 a_2(\theta) H_2\big) dx.
\]

For convergence of $Z(\theta)$ (and integrability of $f_\theta$), we require $2 a_2(\theta) < 1/2$. For our local computations near $\theta=0$, \(c\) can then take values in a wide enough interval.

Write $u(\theta):=\tfrac12 \log f_\theta = a_1 H_1 + a_2 H_2 - \tfrac12 \log Z(\theta) + \log s_0$. Then the score functions are
\[
Y_1 = 2\,\partial_\theta u = 2\big(a_1' H_1 + a_2' H_2 - \tfrac12 Z'(\theta)/Z(\theta) \big),\quad
Y_2 = 2\,\partial_\theta^2 u = 2\big(a_1'' H_1 + a_2'' H_2 - \tfrac12 (Z''/Z - (Z'/Z)^2) \big).
\]

At the base point $\theta=0$, we have
\[
a_1'(0)=1,\quad a_2'(0)=1,\quad a_2''(0)=c,\quad
Z(0)=1,\quad Z'(0)=0,\quad Z''(0)=12,
\]
so that the scores are
\[
Y_1(0) = 2(H_1 + H_2), \qquad Y_2(0) = 2c H_2 - 12.
\]

The jets are
\[
\eta_1 = \partial_\theta s_\theta = \frac12 Y_1 s_\theta,\qquad
\eta_2 = \partial_\theta^2 s_\theta = \frac12 \Big(Y_2 + \frac12 Y_1^2 \Big) s_\theta.
\]

Expanding at $\theta=0$ with the Hermite product identity
\[
H_m H_n = \sum_{k=0}^{\min(m,n)} k! \binom{m}{k}\binom{n}{k} H_{m+n-2k},
\]
we obtain
\begin{align*}
\eta_1(0) &= (H_1 + H_2) s_0,\\
Y_1(0)^2 &= 4(H_1 + H_2)^2 = 4(H_1^2 + 2 H_1 H_2 + H_2^2)\\
&= 4\big[(H_2 + H_0) + 2(H_3 + 2H_1) + (H_4 + 4H_2 + 2H_0)\big] = 4(H_4+2H_3+5H_2+4H_1+3H_0),\\
\eta_2(0) &= \frac12 \Big(2c H_2 -12 + \frac12 Y_1(0)^2 \Big)s_0
= s_0 \Big(H_4 + 2H_3 + (5+c) H_2 + 4 H_1-3H_0\Big).
\end{align*}
The score-image span at order $m=2$ is
\[
\widetilde{\mathcal{T}}_2(0)=\operatorname{span}\{\widetilde{Y}_1(0),\widetilde{Y}_2(0)\}
=\operatorname{span}\{2(H_1+H_2)s_0,\; (2c\,H_2-12H_0) s_0\}.
\]
The $2$-jet span is
\[
\mathcal T_2(0)=\operatorname{span}\{\eta_1(0),\eta_2(0)\}
=\operatorname{span}\Big\{(H_1+H_2)s_0,\;\big(H_4 + 2H_3 + (5+c)H_2 + 4H_1-3H_0 \big)s_0\Big\}.
\]

Define the unbiased estimator $T(X)=H_3(X)$. Then $Z_0:=T-\mathbb E_\theta[H_3(X)]$ satisfies, at $\theta=0$,
\[
Z_0\,s_0 \;=\; H_3 s_0.
\]
At $m=1$, the score span is $\widetilde{\mathcal{T}}_1(0)=\operatorname{span}\{\widetilde{Y}_1(0)\}=\operatorname{span}\{(H_1{+}H_2)s_0\}$, hence
\[
\mathcal B^{(1)}_{\text{score}}=\mathcal B^{(1)}_{\text{jet}}\equiv \mathcal B^{(1)}=\big\|\operatorname{Proj}_{\widetilde{\mathcal{T}}_1(0)}(H_3 s_0)\big\|^2=0,
\]
since $H_3\perp H_1,H_2$. Recall again that
\[
\eta_1=(H_1+H_2)s_0,\qquad
\eta_2=(H_4+2H_3+(5+c)H_2+4H_1-3H_0)s_0.
\]
The projection coefficient is
\[
\Gamma_{11} \;=\; \frac{\langle \eta_2,\eta_1\rangle}{\langle \eta_1,\eta_1\rangle}
,
\]
so that
\[
\Pi_1 \;=\; \eta_2 - \Gamma_{11}\eta_1
= s_0\Big( H_4 + 2H_3 + \tfrac{1+c}{3}H_2 - \tfrac{2(1+c)}{3}H_1-3H_0\Big).
\]

The inner product with $H_3s_0$ is
\[
\langle H_3 s_0, \Pi_1 \rangle = 2\cdot 3! = 12,
\]
and the squared norm of $\Pi_1$ is
\[
\|\Pi_1\|^2 
= 57 + \tfrac{2}{3}(1+c)^2.
\]

Therefore the curvature contribution at order $1$ is
\[
\|\mathrm{Proj}_{\operatorname{span}\{\Pi_1\}}(H_3 s_0)\|^2
= \frac{\langle H_3 s_0,\Pi_1\rangle^2}{\|\Pi_1\|^2}
= \frac{144}{\,57 + \tfrac{2}{3}(1+c)^2\,}.
\]

In particular, we get a strictly positive improvement
\[
\mathcal C^{(1)} \;=\; \mathcal B^{(1)} \;+\; \frac{144}{\,57 + \tfrac{2}{3}(1+c)^2\,},
\]
which is strictly less than the full variance $\|H_3 s_0\|^2=6$ (since $\Pi_1$ is not parallel to $H_3s_0$ for any real~$c$).

At $m=2$, we still get
\[
\mathcal B^{(2)}_{\text{score}}=\big\|\operatorname{Proj}_{\widetilde{\mathcal{T}}_2(0)}(H_3 s_0)\big\|^2=0.
\]
It is immediate that
\[
\mathcal B^{(2)}_{\mathrm{jet}}
\;=\;\| \operatorname{Proj}_{\mathcal T_2}(H_3 s_0)\|^2
\;=\; \mathcal C^{(1)}.
\]
In particular, $\mathcal B^{(2)}_{\mathrm{jet}}>0$ for every real $c$. Since \(\mathcal B^{(2)}_{\text{score}}=0\) and \(\mathcal C^{(2)}\geq \mathcal B^{(2)}_{\mathrm{jet}}>0\) for all \(c\in\mathbb{R}\), we have demonstrated strict improvement also at \(m=2\) with Theorem \ref{thm:curvature-corrected}.

The sufficient conditions stated in Proposition \ref{prop:single-Naug-relations} can be checked in this example with \(m=2\): If \(\operatorname{Proj}_{\mathcal N_m^{\mathrm{aug}}}\widetilde{Z}_0\neq0\) and \(\mathcal N_m^{\mathrm{aug}}\not\subseteq \widetilde{\mathcal{T}}_m\), then strict improvement holds:
  \[
  \mathcal C_{\mathrm{aug}}^{(m)} \;>\; \mathcal B_{\mathrm{score}}^{(m)}.
  \]
Towards this, recall the Bell remainder for $\eta_2$:
\[
\eta_2 \;=\; \tfrac12\widetilde{Y}_2 + W_2,\qquad
W_2 \;=\; \tfrac14 Y_1^2 s_0
= s_0\big(H_4+2H_3+5H_2+4H_1+3H_0\big).
\]
Decompose $W_2$ into its projection onto $\mathcal T_2$ and the normal (residual) part
\[
W_2 = \operatorname{Proj}_{\mathcal T_2}W_2 \;+\; W_2^\perp,\qquad
W_2^\perp := \operatorname{Proj}_{\mathcal T_2^\perp}W_2.
\]

We want to find the residual
\[
W_2^\perp \;=\; W_2 - \operatorname{Proj}_{\mathcal T_2}(W_2),\qquad 
\mathcal T_2=\operatorname{span}\{\eta_1,\eta_2\}.
\]
 
Using \(\langle H_m s_0, H_n s_0\rangle = \delta_{mn}\,n!\), we get
\begin{align*}
\|W_2\|^2
&= 123 \\[4pt]
\|\eta_1\|^2 &=3,\\[4pt]
\|\eta_2\|^2
&
= 2c^2 + 20c + 123,\\[4pt]
\langle \eta_1,\eta_2\rangle &= 4 + 2(5+c) = 14+2c,\\[4pt]
\langle W_2,\eta_1\rangle &=  14,\\[4pt]
\langle W_2,\eta_2\rangle &= 105 + 10c.
\end{align*}
 
Write \(\operatorname{Proj}_{\mathcal T_2}(W_2)=a\,\eta_1+b\,\eta_2\). The coefficients \((a,b)\) solve
\[
\begin{pmatrix}
\|\eta_1\|^2 & \langle \eta_1,\eta_2\rangle \\[0.5ex]
\langle \eta_1,\eta_2\rangle & \|\eta_2\|^2
\end{pmatrix}
\begin{pmatrix} a \\ b \end{pmatrix}
=
\begin{pmatrix} \langle W_2,\eta_1\rangle \\ \langle W_2,\eta_2\rangle \end{pmatrix},
\]
Solving, we obtain the explicit coefficients
\[
a = \frac{8c^2 - 70c + 252}{2c^2 + 4c + 173}, \qquad b = \frac{2c + 119}{2c^2 + 4c + 173}
\]

The residual (normal) component is
\[
W_2^\perp
=
W_2-a\,\eta_1-b\,\eta_2.
\]

Using
\[
\|W_2^\perp\|^2
=
\|W_2\|^2
-
\big(a\langle W_2,\eta_1\rangle
+
b\langle W_2,\eta_2\rangle\big),
\]
together with $\|W_2\|^2=123$, we find
\[
\|W_2^\perp\|^2 = \frac{114c^2 + 72c + 5256}{2c^2 + 4c + 173}.
\]

Since
the $H_3$ coefficient of $W_2^\perp$ is
\[
2-2b=2(1-b),
\]
we obtain
\[
\langle H_3 s_0, W_2^\perp\rangle
=
12(1-b),
\qquad
\langle H_3 s_0, W_2^\perp\rangle^2
=
144(1-b)^2.
\]

\begin{figure}[t]
    \centering
    \includegraphics[width=0.35\textwidth]{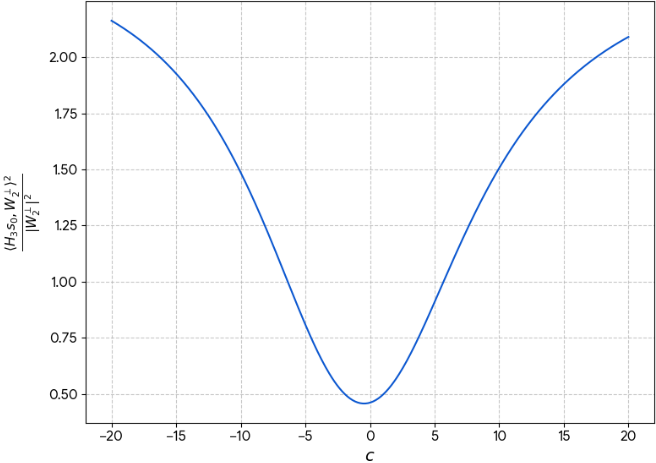}
    \caption{The ratio $\frac{\langle H_3 s_0, W_2^\perp \rangle^2}{\|W_2^\perp\|^2}$ as a function of the parameter $c$.}
    \label{fig:leakage_ratio}
\end{figure}

Therefore,
\[
\frac{\langle H_3 s_0, W_2^\perp \rangle^2}{\|W_2^\perp\|^2} = \frac{96(c^2 + c + 27)^2}{(2c^2 + 4c + 173)(19c^2 + 12c + 876)}.
\]
Note that the residual $W_2^\perp$ has a non-vanishing component in the $H_3$ direction for every \(c\in \mathbb{R}\). With \(\mathcal N_2^{\mathrm{aug}}=\mathrm{span}\{W_2^\perp\}\) and \(\widetilde{\mathcal{T}}_2=\operatorname{span}\{\widetilde{Y}_1(0),\widetilde{Y}_2(0)\}
=\operatorname{span}\{2(H_1+H_2)s_0,\; (2c\,H_2-12H_0) s_0\}\), this implies that the sufficient conditions \(\operatorname{Proj}_{\mathcal N_2^{\mathrm{aug}}}(H_3 s_0)\neq 0\) and \(\mathcal N_2^{\mathrm{aug}}\not\subseteq \widetilde{\mathcal{T}}_2\) are satisfied for all real \(c\), and Proposition \ref{prop:single-Naug-relations} guarantees that
  \(
  \mathcal C_{\mathrm{aug}}^{(2)} \;>\; \mathcal B_{\mathrm{score}}^{(2)},
  \) with the correction along the new normal direction in the augmented space as computed explicitly above. For reader convenience, we provide a visualization of the correction we get by means of a graph that depicts how \(\frac{\langle H_3 s_0, W_2^\perp \rangle^2}{\|W_2^\perp\|^2}\) varies with \(c\) in Figure \ref{fig:leakage_ratio}.
  
\end{example}

Finally, we consider a variation of Example \ref{ex:sat3} to validate our claims in Proposition \ref{prop:single-Naug-relations} regarding possible improvements over \(\mathcal B_{\mathrm{score}}^{(m)}\)
by explicitly using the extrinsic directions \(W_2^\perp, \dots, W_m^\perp\) in addition to \(\mathcal{B}^{(2)}_{\text{jet}}\), serving to compare lower bounds obtained using only the \(\{\eta_k\}_{k=1}^m\) and those that use \(\{\widetilde{Y}_k\}_{k=1}^m\).

\begin{example}[Strict improvement of $\mathcal C_{\mathrm{aug}}^{(2)}$ over 
$\mathcal B_{\mathrm{score}}^{(2)}$]
\label{ex:asym-quartic-aug}

Consider the one-dimensional location model
\[
f(x;\theta)
=
\frac{1}{Z}\exp\!\left(- (x-\theta)^4 - \alpha (x-\theta)^3\right),
\qquad \alpha\neq 0,
\]
where
\[
Z = \int_{\mathbb R} \exp(-u^4-\alpha u^3)\,du < \infty,
\quad u:=x-\theta.
\]
This model is smooth, regular, non-symmetric, and lies outside the exponential family.

\paragraph{Unbiased estimator.}
Since this is a location family, define
\[
\mu := \mathbb E_0[u]
= \frac{1}{Z}\int_{\mathbb R} u\,e^{-u^4-\alpha u^3}\,du,
\]
and consider the estimator
\[
T(X) := X - \mu.
\]
Then
\[
\mathbb E_\theta[T(X)]
=
\theta + \mathbb E_0[u] - \mu
=
\theta,
\]
so $T$ is unbiased for all $\theta$. The centered error is
\[
Z_0 := T(X)-\theta = u-\mu.
\]

\paragraph{Scores and score images.}
The log-density is
\[
\log f(x;\theta) = -u^4 - \alpha u^3 - \log Z,
\]
hence the first two score functions are
\[
Y_1(u) = 4u^3 + 3\alpha u^2,
\qquad
Y_2(u) = -12u^2 - 6\alpha u.
\]
Let $s_\theta=\sqrt{f(\cdot;\theta)}$ and $\widetilde Y_j = Y_j s_\theta$.

\paragraph{Jets via Bell polynomials.}
Using the Bell polynomial identity
\[
\eta_k = \partial_\theta^k s_\theta
= s_\theta\, B_k\!\Big(\tfrac12 Y_1,\dots,\tfrac12 Y_k\Big),
\]
the first two jets are
\begin{align*}
\eta_1
&= \tfrac12 Y_1 s_\theta
= s_\theta\Big(2u^3 + \tfrac32\alpha u^2\Big),\\
\eta_2
&= s_\theta\Big(\tfrac14 Y_1^2 + \tfrac12 Y_2\Big).
\end{align*}
Since
\[
Y_1^2 = 16u^6 + 24\alpha u^5 + 9\alpha^2 u^4,
\]
we obtain
\[
\eta_2
=
s_\theta\Big(
4u^6 + 6\alpha u^5 + \tfrac94\alpha^2 u^4
-6u^2 -3\alpha u
\Big).
\]

\paragraph{Bell remainder and augmented normal direction.}
By definition,
\[
\eta_2 = \tfrac12 Y_2 s_\theta + W_2,
\]
where the Bell remainder is
\[
W_2
=
s_\theta\Big(
4u^6 + 6\alpha u^5 + \tfrac94\alpha^2 u^4
\Big).
\]
Let
\[
\mathcal T_2 = \operatorname{span}\{\eta_1,\eta_2\},
\qquad
\mathcal N_2^{\mathrm{aug}} = \operatorname{span}\{W_2^\perp\}
\subset \mathcal T_2^\perp.
\]

\paragraph{Strict improvement at order $m=2$.}
The estimator residual in $L^2(\mu)$ is
\[
\widetilde Z_0 = (u-\mu)s_\theta,
\]
which has a nonzero odd component. Notice that $W_2$ contains the odd term
$6\alpha u^5 s_\theta$ and $\alpha\neq 0$.

The second-order score-based Bhattacharyya bound is
\[
\mathcal B_{\mathrm{score}}^{(2)}
=
\big\|\operatorname{Proj}_{\widetilde{\mathcal T}_2}\widetilde Z_0\big\|^2,
\qquad
\widetilde{\mathcal T}_2=\operatorname{span}\{\widetilde Y_1,\widetilde Y_2\}.
\]
The augmented bound is
\[
\mathcal C_{\mathrm{aug}}^{(2)}
=
\big\|\operatorname{Proj}_{\mathcal T_2}\widetilde Z_0\big\|^2
+
\big\|\operatorname{Proj}_{\mathcal N_2^{\mathrm{aug}}}\widetilde Z_0\big\|^2.
\]

For generic asymmetric quartic models, the conditions $\operatorname{Proj}_{\mathcal N_2^{\mathrm{aug}}}\widetilde Z_0\neq 0$
and $\mathcal N_2^{\mathrm{aug}}\not\subseteq\widetilde{\mathcal T}_2$ are satisfied, and Proposition~\ref{prop:single-Naug-relations} implies
\[
\mathcal C_{\mathrm{aug}}^{(2)} > \mathcal B_{\mathrm{score}}^{(2)}.
\]
The magnitude of the improvement obviously depends on the asymmetry
parameter $\alpha$. For instance, with \(\alpha=1.95\), we get the numerically verifiable values \(\mathcal B_{\mathrm{score}}^{(2)}\approx 0.16, \quad C_{\mathrm{aug}}^{(2)}\approx 0.19\). 
This strict improvement occurs at fixed order $m=2$ and is due, in this particular example, mainly to the \(\mathcal B_{\mathrm{jet}}^{(2)}\) component in \(\mathcal C_{\mathrm{aug}}^{(2)}\).
\end{example}

\begin{remark}
Our primary goal in these examples is conceptual: to demonstrate that statistical inefficiency can be interpreted geometrically as curvature of the embedded statistical manifold in \(L^2(\mu)\).
The Hermite–Gaussian examples are deliberately chosen for analytical transparency, allowing all geometric quantities (jets, second fundamental form, normal components) to be computed explicitly and rigorously.
While the resulting correction depends on the estimator and may not always be computationally convenient, this dependence is intrinsic: inefficiency itself is estimator-specific.
Importantly, the framework applies beyond toy examples. For instance, constrained or misspecified estimators and semiparametric models with nuisance parameters could give rise to nonzero curvature terms and admit strict refinements over the CRB. Furthermore, we also demonstrated a possible numerical computation of our bounds using Example \ref{ex:asym-quartic-aug} where, in fact, our primary interest was to illustrate Proposition \ref{prop:single-Naug-relations}.
\end{remark}

\section{Conclusion}

In this work we developed (non-asymptotic) refinements of the classical Cramér-Rao lower bound, and some of its own refinements such as the Bhattacharyya-type bounds that use higher-order score functions, in a manner consistent with the local differential geometry of the statistical model. The CRB case incorporated curvature-aware corrections derived from the second fundamental form on tangent vectors of the statistical manifold embedded into a fixed Hilbert space using the square root map, while the higher-order version also relied on a natural extension. Furthermore, we considered the decomposition of the corresponding jets in terms of the true scores using a special case of the Faà di Bruno formula to demonstrate improvements over Bhattacharyya-type bounds that use log-likelihood derivatives. 
We detailed explicit examples where the geometric correction terms result in definite improvements over the CRB and its variants. This confirms that curvature-aware corrections, when properly projected and interpreted within a Riemannian framework, yield provable and relevant inefficiency measures.
Further research may examine the generality and limitations of these geometric corrections in broader model classes, such as those with higher-order nonlinearity in the statistical embedding or with weaker assumptions on regularity than what we allowed, for instance. The connection between extrinsic curvature and information-theoretic optimality, particularly beyond the second-order regime, remains an area with substantial theoretical and methodological potential.

\paragraph{Acknowledgements.} We are thankful to Prof. Pillai for sharing his article \cite{pillai}, which formed the initial motivation for this work.

\bibliography{sn-bibliography}
\end{document}